\documentclass[12pt]{article}

\usepackage{geometry,amssymb,amsthm,amsmath,
graphics,enumerate}

\usepackage{xcolor}
\usepackage{times}
\usepackage{amsfonts}
\usepackage{amscd}
\usepackage{url}
  \def\mcolon{\! : \!}

  \usepackage[normalem]{ulem}
\usepackage{soul}

\newcommand\myrestriction{\mathord\restriction}
\def\mr#1{\myrestriction_{#1}}

\newbox\smilebox
\newbox\anchorbox
\newbox\noanchorbox
\newbox\tempbox

\setbox\smilebox=\hbox{$\smile$}

\def\anchor{\hbox{\vtop{
           \hbox to \wd\smilebox{\hfil\vrule width.4pt height7pt depth1pt\hfil}
           \vskip  -11.5truept
           \hbox to \wd\smilebox{\hfil$\smile$\hfil}}}}
\setbox\anchorbox=\anchor
\def\noanchor{\hbox{\vtop{
           \hbox to \wd\anchorbox{\hfil\anchor\hfil}
           \vskip -14truept
           \hbox to \wd\anchorbox{\hfil/\hfil}}}}
\setbox\noanchorbox=\noanchor

\def\fg#1#2#3{\setbox\tempbox=\hbox{$\scriptstyle{#2}$}
\ifnum\wd\anchorbox>\wd\tempbox\dimen255=\wd\anchorbox
\else\dimen255=\wd\tempbox\fi
{#1\,\vtop{\hbox to \dimen255{\hfil\anchor\hfil}
           \vskip -6truept
           \hbox to \dimen255{\hfil$\scriptstyle{#2}$\hfil}}
           \,#3}}

\def\nfg#1#2#3{\setbox\tempbox=\hbox{$\scriptstyle{#2}$}
\ifnum\wd\noanchorbox>\wd\tempbox\dimen255=\wd\noanchorbox
\else\dimen255=\wd\tempbox\fi
{#1\,\vtop{\hbox to \dimen255{\hfil\noanchor\hfil}
           \vskip -6truept
           \hbox to \dimen255{\hfil$\scriptstyle{#2}$\hfil}}
           \,#3}}

\setbox1=\hbox{$\bot$}

\def\north#1#2{#1\,
\hbox{$\bot$\llap {\hbox to\wd1 {\hfil $/$\hfil}}}
\,#2}

\def\nao#1#2#3{#1\  \hbox{\vtop{
\baselineskip=4pt
\hbox{$\bot$\llap {\hbox to\wd1 {\hfil $/$\hfil}}
\hskip .05em \llap{\hbox{$^{\scriptscriptstyle{a}}$}}}\hbox{$\scriptstyle
{#2}$}}}\, #3}

\def\abar{\overline{a}}
\def\bbar{\overline{b}}
\def\cbar{\overline{c}}
\def\dbar{\overline{d}}
\def\ebar{\overline{e}}

\def\gbar{\overline{g}}
\def\hbar{\overline{h}}

\def\wbar{\overline{w}}
\def\xbar{\overline{x}}
\def\ybar{\overline{y}}
\def\zbar{\overline{z}}

\def\pcl{{\rm acl}}
\def\tp{{\rm tp}}

\def\A{{\cal A}}
\def\B{{\cal B}}

\def\CC{{\cal C}}
\def\D{{\cal D}}
\def\E{{\cal E}}
\def\FF{{\bf F}}

\def\F{{\cal F}}
\def\FF{{\bf F}}
\def\II{{\bf I}}

\def\L{{L_{{\rm ord}}}}
\def\SS{{\cal  S}}

\def\P{{\cal P}}
\def\Q{{\mathbb Q}}
\def\QQ{{\mathbb Q}}

\def\Z{{\mathbb Z}}

\def\tp{{\rm tp}}

\def\dom{{\rm dom}}
\def\pcl{{\rm pcl}}

\def\Fa0{{\FF^a_{\aleph_0}}}

\def\<{\langle}
\def\>{\rangle}

\def\P{{\bf P}}
\def\ee{{\bf e}}
\def\At{{\bf At_T}}

\def\pcl{{\rm pcl}}
\def\rk{{\rm rk}}
\def\Gdot{{\bf \tilde{G}}}
\def\phi{\varphi}

\def\rk{{\rm rk}}
\newtheorem{Theorem}{Theorem}[section]
\newtheorem{Proposition}[Theorem]{Proposition}
\newtheorem{Definition}[Theorem]{Definition}
\newtheorem{Remark}[Theorem]{Remark}
\newtheorem{Example}[Theorem]{Example}
\newtheorem{Lemma}[Theorem]{Lemma}
\newtheorem{Corollary}[Theorem]{Corollary}
\newtheorem{Fact}[Theorem]{Fact}
\newtheorem{Data}[Theorem]{Data}

\newtheorem{Construction}[Theorem]{Construction}
\newtheorem{Claim}[Theorem]{Claim}

\begin{document}

\title{An analogue of $U$-rank for atomic classes}

\author{John Baldwin\thanks{ Research partially supported by Simons travel grant G3535.} \\University of Illinois, Chicago\and  Chris Laskowski\\University of Maryland\thanks{Partially supported
by NSF grants DMS-1855789 and DMS-2154101.}\and \\ Saharon Shelah\\Hebrew University\thanks{NSF-BSF 2021: grant with M. Malliaris, NSF 2051825, BSF 3013005232 (2021-10-2026-09)
Rutgers 2018 DMS 1833363: NSF DMS Rutgers visitor program (PI S. Thomas) (2018-2022).
Publication 1183 of the  Shelah archive.}}

\date{\today}

\maketitle

\begin{abstract}
For a countable, complete, first-order theory $T$, we study $\At$, the class of atomic models of $T$.
We develop an analogue of $U$-rank and prove two results.  On one hand, if some $\tp(d/a)$ is not ranked,
then there are $2^{\aleph_1}$ non-isomorphic models in $\At$ of size $\aleph_1$.  On the other hand, if all types
have finite rank, then the rank is fully additive and 
every finite tuple is dominated by an independent set of realizations of pseudo-minimal types.
\end{abstract}

For a countable, complete first order theory $T$, a model $M$ is {\em atomic} if $\tp(a)$ is principal, i.e., is generated by a complete formula for every finite tuple $a$ from $M$.
In  this paper, we continue our investigations of dichotomies among classes $\At$ of atomic models of a countable, complete, first-order theory $T$.  
One reason for studying such classes relates to complete sentences of $L_{\omega_1,\omega}$.
It is well known, see e.g., \cite[\S 6]{Baldwincatmon} that for every complete $L_{\omega_1,\omega}$-sentence $\Phi$, there is a complete first-order theory $T$ in a possibly larger countable
language such that the reducts of models of $\Phi$ 
are in $1$-$1$ correspondence with the atomic  models of $T$.

We wish to develop a classification theory for atomic classes $\At$ akin to the work of the third author concerning $Mod(T)$ for complete, first order $T$.
In the first-order context, a fundamental dividing line is {\em superstability}.  The third author proved that if $T$ is unsuperstable, then $Mod(T)$ contains $2^\kappa$ non-isomorphic
models of size $\kappa$ for each uncountable cardinal $\kappa$.  On the other hand, if $T$ is superstable, then models of $T$ admit a desirable independence relation, non-forking.
From this, one can measure the forking complexity of a type by assigning ranks to the space of types, e.g., $R^\infty(p)$ or
$U(p)$.  These ranks allow one to prove structural results for $Mod(T)$ by way of inductive arguments on the space of types.

When one is only considering the class $\At$ of atomic models of $T$, the usual dividing lines are not relevant and the test questions need to be altered.
It is notable that in the context of atomic models, even (first order) stability is not relevant.  See, for example, Example~\ref{ranked}.  
There, $T=Th(N)$ is not stable in the first order context, yet $\At$ has a unique atomic model in every infinite cardinality.  
Worse, as the Upward L\"owenheim-Skolem theorem can fail for atomic models, asking for many atomic models in all uncountable cardinalities  may well be meaningless; e.g. if there are models only up to $\aleph_1$.
However, by classical results of Vaught, an atomic model of size $\aleph_1$ exists if and only if the (unique) countable atomic model is not minimal.
Thus, it is natural to call an atomic class $\At$ {\em unstructured} if it contains $2^{\aleph_1}$ non-isomorphic atomic models, each of size $\aleph_1$,
and then to ask what effect does `structured' have on its countable models. 

In this paper, we  introduce and develop a rank, $\rk(d/a)$, on all finite tuples $d,a$ from a fixed countable atomic model $N$ and with Theorem~\ref{bigone} we prove that
if if $\At$ is structured (fewer than $2^{\aleph_1}$ non-isomorphic atomic models of size $\aleph_1$) then $\rk(d/a)$ exists for all $d,a\subseteq N$.  
Our rank $\rk$ is similar to $U$-rank, which, in the first order context, is the foundation rank on the space of complete types, which are tree ordered by the relation $p<q$ iff
$q$ is a forking extension of $p$.   
In first order, a theory is superstable if and only if $U(p)$ is ordinal valued for every complete type $p$. 
 
Our rank can also be viewed as a foundation rank of types $\tp(d/a)$ with respect to the relation of `an extension making some element pseudo-algebraic.' 
However, because $\rk(d/a)$ is determined by $\tp(d/a)$, which is generated by a complete formula, we additionally get that our rank is continuous, which is not generally true
of $U$-rank in first order, superstable theories.    
For such theories, an alternate rank is  $R^\infty$-rank, but it is a rank on formulas as opposed to types.  Its natural generalization to types, given by $R^\infty(p)=\min\{R^\infty(\theta):\theta\in p\}$ is only semi-continuous.   
It is pleasing that our rank  possesses both of the desirable properties of $R^\infty$-rank and $U$-rank -- it is a rank on types that is fully  continuous and  is a foundation rank
of a natural extendibility property.

The new rank is defined in Section~2, with the salient features developed in Sections 3 and 4. 
The main results are stated explicitly in Theorem~\ref{summary}, but here is a summary.  
We prove semantic equivalents of the rank in terms of the existence of certain chains, which shows that $\rk(d/a)=0$ and $\rk(d/a)=1$ are equivalent to being pseudo-algebraic and
pseudo-minimal, respectively.    
Continuing upward, we see that `having rank $n<\omega$ is extendible,' i.e., if $\rk(d/a)=n$, then for every type $q\in S_{at}(a)$ there is some $b$ realizing $q$ with $\rk(d/ab)=n$.
This result implies  that among finitely ranked atomic classes, the rank is fully additive.  
Using this additivity, 
we conclude that any finite set is dominated in some sense by an independent sequence of psuedo-minimal types.    Collectively, these results show that finitely ranked atomic classes $\At$ are similar to finitely ranked superstable theories in the first order context.  
Finally, in Section~5,  we prove our main result,  Theorem~\ref{bigone}.  
Its rather lengthy argument shows   that the assumption of `few atomic models in $\aleph_1$' implies that 
the class $\At$ is ranked.   Even though the theorem is proved  in ZFC, heavy use is made of certain forcing constructions.

\section{Context} En route to proving a first order theory $T$ is categorical in $\aleph_1$ is categorical in all uncountable
cardinalities, Morley  \cite{Morley65}, Morley exploited the upwards L\"{o}wenheim-Skolem theorem.
He applied the Erd\"{o}s-Rado theorem to deduce that the unique model in $\aleph_1$ can be represented as
an Ehrenfeucht-Mostowski over the cardinal $\aleph_1$ and concluded that the theory admitted only countably
many types over a countable set. He called this property {\em totally transcendental}, now usually called $\omega$-stable. Shelah later extended this result  under the set theoretic hypothesis
$2^{\aleph_n}< 2^{\aleph_{n+1}}$, by showing  a complete $L_{\omega_1,\omega}$-sentence that is categorical  in all infinite cardinals below
$\aleph_\omega$  is {\em excellent} and consequently, has arbitrarily large models  and is categorical in all uncountable
cardinalities.

As described in \cite{BLSmanymod}, \cite{Sh87a}, and \cite[Chapter 6]{Baldwincatmon}, the models of such a sentence can be thought of as the atomic models (all finite sequences in each model
realize a principal type) of a first order theory; we work here with that assumption.

\cite{Sh87a,Shaec1} proved (See also \cite[Chapter 17]{Baldwincatmon}):
\begin{Fact}[Martin's Axiom] There is a sentence $\psi$ in
$L(Q)$ with the joint embedding property that is
$\kappa$-categorical for every $\kappa < 2^{\aleph_0}$. In ZFC one
can prove $\psi$ is  $\aleph_0$-categorical but the associated AEC
has neither the amalgamation property in $\aleph_0$ nor is
$\omega$-stable.
\end{Fact}

$L(Q)$ is first order logic extended by the quantifier, `there exists countably many'. Shelah
proposed a variant to get such a counterexample in $L_{\omega_1,\omega}$;  however, there was a gap.
This article is part of a more than 20 year effort to fill the gap or, mostly, to prove no such example 
exists by showing $\aleph_1$-categoricity implies $\omega$-stability; the existence of a model in $\beth_1^+$
is state of the art in that direction \cite{BLSwhendoes}. This particular paper is in the midst of those 
alternatives: establishing the existence of a rank analogous to the $U$-rank in superstable theories from
the assumption that there are few (atomic)  models in $\aleph_1$.

%

The hope was to show that any ranked sentence has a model in the continuum, but sadly that goal remains unattained.

\section{The new rank and statements of the main results}

Throughout this section, fix a complete theory $T$ in a countable language and assume there is a countable  atomic model $N$ that is not minimal.
So long as we restrict to complete types
 of finite tuples over finite subsets of $N$, $N$ serves as a `monster model'.   For every
finite $a$, $N$ realizes all types over $a$,
 and moreover $N$ is homogeneous -- if $a,b,c$ are finite tuples from $N$ with $\tp(a/c)=\tp(b/c)$, then there is an automorphism $\sigma\in Aut(N)$
fixing $c$ pointwise
with $\sigma(a)=b$.

\medskip

 {\bf Throughout Sections 2-4,  we assume all finite tuples are from this  countable, atomic model $N$.}
 
\medskip

We recall a definition from \cite[\S 2]{BLSmanymod}.

\begin{Definition} \label{pcl} {\em  We say {\em $d$ is in the pseudo-closure of $a$}, $d\in\pcl(a)$, if   every model $M\preceq N$ containing $a$ also contains $d$.
For a finite tuple $a$, $\pcl(a)=\{d\in N:d\in\pcl(a)\}$ and for
$A\subseteq N$ any set, $\pcl(A)=\bigcup\{\pcl(a): a\in A$\  finite$\}$.

A complete type $p\in S(a)$ is {\em pseudo-algebraic} if $d\in\pcl(a)$ for some (equivalently every) $d$ realizing $p$.
}
\end{Definition}

As we are assuming $N$ has a proper elementary substructure, psuedo-closure is not degenerate, i.e., $N\neq\pcl(\emptyset)$.

\begin{Definition}  {\em 
In \cite{BLSmanymod}, a type $p=\tp(d/a)$ is {\em pseudo-minimal} if $d\not\in\pcl(a)$ and for every $b,c$,
if $c\in\pcl(abd)\setminus\pcl(ab)$, then $d\in\pcl(abc)$.

We say that the {\em pseudo-minimal types are dense} if, for every non-pseudo-algebraic $p=\tp(d/a)$,
there is some $a^*$ such that $\tp(d/aa^*)$ is pseudo-minimal.
}
\end{Definition}

\begin{Definition}\label{bfP}  {\em  Let $\P$ denote the set of all types $\tp(d/a)$ for finite tuples $a,d\subseteq N$.   As $N$ is atomic, every $p\in\P$ is principal.

We define a rank
$\rk : \P\rightarrow{\bf ON}\cup\{\infty\}$ by induction on $\alpha$ requiring for any finite sequence $a$
:
\begin{itemize} \item $\rk(d/a) \ge 0$
 \item  $\rk(d/a)
\ge \alpha>0$ if and only if for every $r(y)\in S_{at}(a)$ and for every $\beta<\alpha$, there exist tuples  $a',b,c$ from $N$ such that
\begin{enumerate}
\item  $\tp(a'/a)=r$;
\item  $\rk(\tp(d/aa'bc))\ge\beta$; and
\item  $c\in\pcl(daa'b)\setminus\pcl(aa'b)$.
\end{enumerate}
\item  For an ordinal $\alpha$, $\rk(p)=\alpha$ if $\rk(p)\ge \alpha$, but $\rk(p)\not\ge\alpha+1$.
\item  Call $\At$ {\em ranked} if $\rk(p)$ is ordinal-valued for every $p\in\P$.
\item  Call $\At$ {\em finitely ranked} if $\rk(p)<\omega$ for every $p\in \P$.  
\end{itemize}
}
\end{Definition}

Observe that  $d\in \pcl(a)$ if and only if $\rk(d/a) =0$.  While the ranks are on {\em complete} formulas, not all formulas
have been ranked. To remedy this we can define $\rk(\phi(x,a)) = \sup\{\rk(p): \phi(x,a) \in p \in \P\}$.  
The following example may give some intuition.  
\begin{Example}  \label{ranked}
{\em
Let $L=\{A,B,\pi,\le
\}$ and 
let $N$ be the $L$-structure where $A$ and $B$ partition
	the universe with $B$ infinite,
	$\pi:A\rightarrow B$ is a total surjective
	function and $(\pi^{-1}(b),\le)\cong (\Z,\le)$, with $ a\not\le a'$ whenever $\pi(a)\neq\pi(a')$.  
	Then $N$ is an atomic model of $T=Th(N)$, and any $M\models T$ will be atomic if and only if $(\pi^{-1}(b),\le)\cong (\Z,\le)$ for every $b\in B$.

	Now choose elements $a,b\in N$  such that $\pi(a)=b$.
	Clearly, $a$ is not
algebraic	over $b$ in the classical sense, however they are ``equi-pseudo-algebraic" i.e., $b\in\pcl(a)$ (trivially) and $a\in\pcl(b)$.

In terms of ranks,
note that  $\pcl(\emptyset) = \emptyset$, so for any $e\in N$, $\rk(e/\emptyset) \geq 1$.
For $a,b$ with $\pi(a)=b$, 
 $rk(a/\emptyset)=\rk(b/\emptyset)=1 $ and both of these types are pseudo-minimal. 
However, $\rk(a/b)=\rk(b/a)=0$.     
Here, $\At$ is categorical in every infinite power and is finitely ranked.
To give an example of a structure of rank 2, add an equivalence relation $E$ and insist that each class is a model
of the current example. 
}

\end{Example}

We will prove four main results about this rank:

\begin{Theorem}  \label{summary}  Let $T$ be a complete theory in a countable language for which there is an uncountable atomic model.  Then:
\begin{enumerate}
\item  (Proposition~\ref{charinfty}) 
For $p=\tp(d/a)\in\P$, $\rk(p)=\infty$ if and only if there is an infinite sequence of models $M_0\preceq M_1\preceq\dots\preceq N$  with $a\in M_0$
and  tuples
$c_n\in M_{n+1}$ such that $c_n\in\pcl(M_nd)\setminus M_n$ for each $n$.
\item  (Theorem~\ref{bigone}) If $\At$ has $<2^{\aleph_1}$ non-isomorphic atomic models of size $\aleph_1$, then $\At$ is ranked.
\item  If $\At$ is ranked, then: (Corollaries~\ref{dense},\ref{extension}, and Proposition~\ref{additivity})
\begin{enumerate}
\item  The pseudo-minimal types are dense;
\item  If $\rk(d/a)=n<\omega$ then there is a model $M\supseteq a$ with $\rk(d/M)=n$;
\item  Among types of finite rank, the rank is fully additive, i.e., $\rk(de/a)=\rk(d/ea)+\rk(e/a)$ whenever $\rk(de/a)<\omega$
\end{enumerate}
\item (Proposition~\ref{domexistence})
 If $\At$ is finitely ranked, then for every pseudo-minimal $\theta(x)$, for every independent tuple $\cbar\in\theta(N)^n$ and every finite $b\subseteq N$,
there is a finite $h\subseteq N$ for which $\cbar$ $\theta$-dominates $b$ over $h$ (Definition~\ref{domdef}).

\end{enumerate}
\end{Theorem}

Note that the main result from \cite{BLSmanymod}, that failure of density of pseudo-minimal types implies  the existence of $2^{\aleph_1}$
non-isomorphic atomic models of size $\aleph_1$, follows immediately from Theorem~\ref{summary}.

\section{Properties of the rank, chains, and additivity}
 We first record some properties of the new rank. 
We begin with two easy monotonicity results.

\begin{Lemma} \label{monotonebase}  Suppose $d_0\subseteq d$ and $a\subseteq a^*$ are from $N$.
Then:
\begin{enumerate}
\item  $\rk(d_0/a)\le\rk(d/a)$; and
\item  $\rk(d/a^*)\le \rk(d/a)$.
\end{enumerate}
\end{Lemma}

\begin{proof} (1) is easy and is left to the reader.  For (2), 
we prove that for any ordinal $\alpha$, $\rk(d/a^*)\ge\alpha+1$ implies $\rk(d/a)\ge\alpha+1$, which suffices.
Suppose $\rk(d/a^*)\ge\alpha+1$.  Choose any $r\in S_{at}(a)$ and any realization $a_0$ of $r$.
Let $r^*:=\tp(a_0/a^*)$.  As $\rk(d/a^*)\ge\alpha+1$, choose a realization $a'$ of $r^*$ and $b,c$ such that
$\rk(d/a^*a'bc)\ge\alpha$, $c\in\pcl(a^*a'bd)\setminus\pcl(a^*a'b)$.
Put $b':=a^*b$.  As $a'b'=a'a^*b$ (as sets),  $a',b',c$ satisfy $a'\models r$, $\rk(d/a'bc)\ge\alpha$,  and $c\in\pcl(a'b'd)\setminus\pcl(a'b')$.
Thus, $\rk(d/a)\ge\alpha+1$.
\end{proof}

The following two Lemmas establish that the rank is continuous.

\begin{Lemma}   \label{betaplus}   For any ordinal $\alpha$, if $\rk(d/a)=\alpha$, then for all $\beta<\alpha$ there is $\beta^*$, $\beta\le\beta^*<\alpha$
and $e\subseteq N$ such that $\rk(d/ae)=\beta^*$.
\end{Lemma}

\begin{proof}  First, since $\rk(d/a)\not\ge \alpha+1$, choose some $r\in S_{at}(a)$ for which there do not exist
$a',b,c$ from $N$ such that $a'$ realizes $r$, $\rk(d/aa'bc)\ge \alpha$ and $c\in\pcl(aa'bd)\setminus\pcl(aa'b)$.
Now  choose $\beta<\alpha$.  By the definition of $\rk(d/a)\ge\alpha$ applied to $\beta$ and the $r$ from above,
there are  $a',b,c$ such that $a'$ realizes $r$, $c\in\pcl(aa'bd)\setminus\pcl(aa'b)$, and $\rk(d/aa'bc)\ge\beta$.
Take $e:=a'bc$ and $\beta^*:=\rk(d/ae)$.
\end{proof}

\begin{Lemma} \label{continuousrank}   For any ordinal $\alpha$, if $\rk(d/a)=\alpha$, then for every $\gamma<\alpha$ there is some $b\subseteq N$ such that
$\rk(d/ab)=\gamma$.
\end{Lemma}

\begin{proof}  We prove this by induction on $\alpha$.  For $\alpha=0$ there is nothing to prove, so fix $\alpha>0$ and assume the statement holds for all $\beta<\alpha$.
Choose any $\gamma<\alpha$ and apply Lemma~\ref{betaplus} to get $e\subseteq N$ such that $\rk(d/ae)\in [\gamma,\alpha)$.   If $\rk(d/ae)=\gamma$ we are done.   Otherwise,
apply the inductive hypothesis to $\rk(d/ae)$ to get $b^*$ such that $\rk(d/aeb^*)=\gamma$.  Then $b:=eb^*$ is as required.
\end{proof}

\begin{Proposition}  \label{Prop2}
\begin{enumerate}
\item  There is some countable ordinal $\gamma^*$ such that, for all finite $d,a\subseteq N$, if
$\rk(d/a)\ge\gamma^*$, then $\rk(d/a)=\infty$.
If $\rk(d/a)\ge\omega_1$, then $\rk(d/a)=\infty$.
\item  If $\rk(d/a)=\infty$, then 
\begin{enumerate}
\item  For any $r\in S_{at}(a)$ there is $a'$ realizing $r$ with $\rk(d/aa')=\infty$; and
\item  There are $b^*,c\subseteq N$ with $c\in\pcl(ab^*d)\setminus \pcl(ab^*)$ with $\rk(d/ab^*c)=\infty$.
\end{enumerate}
\end{enumerate}
\end{Proposition}

\begin{proof}  (1)  Let $I$ be the image of the rank function $\rk:\P\rightarrow ON\cup\{\infty\}$.
By Lemma~\ref{continuousrank} $I\setminus\{\infty\}$ is downward closed and is countable since $\P$ is.
Thus, $I\setminus\{\infty\}=\gamma$ for some countable ordinal $\gamma$.  Then $\gamma^*:=\gamma+1$ satisfies (1).

(2)  Suppose $\rk(d/a)=\infty$.  Then $\rk(d/a)\ge \gamma^*+1$ for $\gamma^*$ as in (1).  It follows from the definition of $\rk(d/a)$
that for every $r\in S_{at}(a)$ there are $a',b,c$ in $N$ such that $a'$ realizes $r$, $c\in\pcl(aa'bd)\setminus \pcl(aa'b)$, and $\rk(d/aa'bc)\ge\gamma^*$.
Then $\rk(d/aa')\ge\rk(d/aa'bc)\ge\gamma^*$, so $\rk(d/aa')=\infty$, satisfying (2a).  For (2b), take $b^*:=a'b$.
\end{proof}

\subsection{Finite and infinite chains}

We begin with a notational extension of our rank function and show that $\rk(d/a)\ge n$ can be characterized by a certain finite chain of submodels of $N$. 

\begin{Definition}\label{extrank}  {\em  Suppose $A\subseteq N$ is infinite.    For $d\subseteq N$ finite, say $\rk(d/A):=\min\{\rk(d/a):a\subseteq A\ \hbox{finite}\}$.
In particular, $\rk(d/A)=\infty$ if and only if $\rk(d/a)=\infty$ for all finite $a\subseteq A$.
}
\end{Definition}

\begin{Proposition} \label{chain}  For all $n\in\omega$, for all finite $d,a\subseteq N$,
$\rk(d/a)\ge n$ if and only if there is a chain $M_0\preceq M_1\preceq\dots\preceq M_n=N$ with $a\subseteq M_0$
and $\pcl(M_id)\cap(M_{i+1}\setminus M_i)\neq \emptyset$ for all $i<n$.
\end{Proposition}

\begin{proof}  By induction on $n$.  This is trivial when $n=0$, so assume the result for $n$.
First, assume $\rk(d/a)\ge n+1$.  Take $r(y):=`y=a'$, $\beta=n$, and choose $a',b,c$ such that $a'$ realizes $r$ (hence $a'=a$),
$\rk(d/abc)\ge n$, $c\in\pcl(abd)\setminus\pcl(ab)$.   Apply the inductive hypothesis to $\tp(d/abc)$ to get
$M_0\preceq\dots\preceq M_n=N$ with $abc\subseteq M_0$ and $\pcl(M_id)\cap(M_{i+1}\setminus M_i)\neq \emptyset$ for all $i<n$.
Since $c\not\in\pcl(ab)$, there is $M_{-1}\preceq M_0$ with $ab\subseteq M_{-1}$, but $c\in M_0\setminus M_{-1}$.  As $c\in\pcl(abd)$, we have
$c\in\pcl(M_{-1}d)$, so $M_{-1}\preceq M_0\preceq \dots \preceq M_n=N$ is a requisite chain of length $n+1$.

Conversely, assume a chain $M_0\preceq \dots\preceq M_{n+1}=N$ satisfies $a\subseteq M_0$ and $\pcl(M_id)\cap(M_{i+1}\setminus M_i)\neq \emptyset$ for all $i<n+1$.
To see that $\rk(d/a)\ge n+1$, choose any $r\in S_{at}(a)$.  Choose any $a'\in M_0$ realizing $r$.  From our assumption on the chain, choose $c\in M_1\setminus M_0$
with $c\in\pcl(M_0d)$.  Choose a finite $b\subseteq M_0$ such that $c\in\pcl(aa'bd)$.  As $c\not\in M_0$, $c\not\in\pcl(aa'b)$.
However, $aa'bc\subseteq M_1$ and the $n$-chain $M_1\preceq\dots\preceq M_n=N$ satisfies $\pcl(M_id)\cap(M_{i+1}\setminus M_i)\neq \emptyset$ for all $1\le i<n+1$.
Thus, $\rk(d/aa'bc)\ge n$ by our inductive hypothesis.  
The above witnesses that $\rk(d/a)\ge n+1$, so we are done.
\end{proof}

Proposition~\ref{chain} has many corollaries.  The first indicates that the adjectives of pseudo-algebraic and pseudo-minimal occur naturally.

\begin{Corollary} \label{nomenclature} Suppose $d,a\subseteq N$ are finite sets.   Then:
\begin{enumerate}
\item  $\rk(d/a)=0$ if and only if $\tp(d/a)$ is pseudo-algebraic; and
\item  $\rk(d/a)=1$ if and only if $\tp(d/a)$ is pseudo-minimal.
\end{enumerate}
\end{Corollary}

\begin{proof}   (1)  Note that $\rk(d/a)=0$ iff  $\rk(d/a)\not\ge 1$ iff there does not  exist $M\preceq N$ with $a\subseteq M$, $d\not\subseteq M$
 iff $\tp(d/a)$ is pseudo-algebraic.

(2)  First, suppose $\rk(d/a)=1$.  By (1), $\rk(d/a)$ is not pseudo-algebraic.  To see that $\tp(d/a)$ is pseudo-minimal, choose any $c\in\pcl(da)\setminus\pcl(a)$,
and assume by way of contradiction that $d\not\in\pcl(ac)$.
Since $d\not\in \pcl(ac)$, choose $M_1\preceq N$ with $ac\subseteq d$, but $d\not\subseteq M_1$.  Also, since $c\not\in\pcl(a)$, there is $M_0\preceq M_1$
with $a\subseteq M_0$, but $c\not\subseteq M_0$.  Then the 2-chain $(M_0,M_1,N)$ witnesses that $\rk(d/a)\ge 2$.
Conversely, assume $\tp(d/a)$ is pseudo-minimal and we show there cannot be a 2-chain $M_0\preceq M_1\preceq N$ with $a\subseteq M_0$ and $c\in M_1\setminus M_0$
with $c\in\pcl(M_0d)\setminus M_0$.   If there were, then taking any finite $b\subseteq M_0$ for which $c\in\pcl(bd)$, the elements $a,b,c,d$ contradict the pseudo-minimality of
$\tp(d/a)$.
\end{proof}

\begin{Corollary}   \label{dense} If $\At$ is ranked, then the pseudo-minimal types are dense.
\end{Corollary}

\begin{proof}  Choose any $d,a\subseteq N$ such that $\tp(d/a)$ is not pseudo-algebraic.  Since $\rk(d/a)$ exists, by Corollary~\ref{nomenclature}, $\rk(d/a)\ge 1$.  
By Lemma~\ref{continuousrank} there is $b\subseteq N$ such that $\rk(d/ab)=1$, hence $\tp(d/ab)$ is pseudo-minimal.
\end{proof}

Finally, we see that finitely ranked types can be extended to types over models with the same rank.

\begin{Corollary}  \label{extension}  Suppose $\rk(d/a)=n<\omega$.  Then:
\begin{enumerate}
\item  For any $r\in S_{at}(a)$ there is $a'$ realizing $r$ such that $\rk(d/aa')=n$.
\item  There is a model $M\preceq N$ with $a\subseteq M$ and $\rk(d/M)=n$.
\end{enumerate}
\end{Corollary}

\begin{proof}  (1) Using Proposition~\ref{chain}, choose an $n$-chain $M_0\preceq\dots M_n=N$ with $a\subseteq M_0$ and $\pcl(M_id)\cap(M_{i+1}\setminus M_i)\neq \emptyset$ for all $i<n$.  Choose any $a'\in r(M_0)$.   Then the same chain demonstrates that $\rk(d/aa')\ge n$.  Hence $\rk(d/aa')=n$ by Lemma~\ref{monotonebase}.

(2) follows immediately by iterating (1) $\omega$ times.  
\end{proof}

\subsection{Additivity}

\begin{Lemma}  \label{choice}  Suppose $X$ is finite and $\tp(d/X)<\omega$.  If $c\in \pcl(dX)\setminus \pcl(X)$, then
$\rk(d/Xc)<\rk(d/X)$.
\end{Lemma}

\begin{proof}  Say $\rk(d/Xc)=n$ and choose an $n$-chain $M_0\preceq\dots M_n=N$ with $Xc\subseteq M_0$ and 
$\pcl(M_id)\cap(M_{i+1}\setminus M_i)\neq \emptyset$ for all $i<n$.    Since $c\in M_0$ but $c\not\in\pcl(X)$, choose $M_{-1}\preceq M_0$
with $X\subseteq M_{-1}$ and $c\in M_0\setminus M_{-1}$.  This gives an $(n+1)$-chain for $\tp(d/X)$, hence $\rk(d/X)\ge n+1$.
\end{proof}

Note that the same result holds when $X$ is infinite as well (so long as the ranks are finite).  

\begin{Proposition}  \label{additivity}   Suppose $a,d,e\subseteq N$ are finite.
\begin{enumerate}
\item  If $\rk(e/a)=k$ and $\rk(d/ae)=\ell$ are both finite, then $\rk(de/a)=k+\ell$, so in particular is finite.
\item  If $\rk(de/a)<\omega$, then $\rk(de/a)=\rk(e/a)+\rk(d/ae)$.
\end{enumerate}
\end{Proposition}

\begin{proof}  (1)  Choose any $n\in\omega$ such that $\rk(de/a)\ge n$ and we argue that $n\le k+\ell$.
To see this, choose an $n$-chain $M_0\preceq\dots\preceq M_n=N$ with $a\subseteq M_0$ and $\pcl(M_ide)\cap(M_{i+1}\setminus M_i)\neq \emptyset$ for all $i<n$. 
We argue by induction on $i$ that 
$$\rk(e/M_{n-i})+\rk(d/M_{n-i}e)\ge i$$
for all $0\le i\le n$.  For $i=0$ this is obvious, so assume it holds for $i<n$ and we show this for $i+1$.  Let $j=n-i-1$.
Choose $c\in (M_{j+1}\setminus M_{j})\cap\pcl(de/M_{j})$.
There are two cases.  If $c\in\pcl(eM_{j})\setminus M_j$ then by Lemma~\ref{choice}, $\rk(e/M_{j})>\rk(e/M_{j+1})$.
On the other hand, if $c\in\pcl(deM_j)\setminus \pcl(eM_j)$, then $rk(d/eM_j)>\rk(d/eM_{j+1})$.
In either case, the sum is incremented by at least one.

However, by Lemma~\ref{monotonebase}, $\rk(e/M_0)\le \rk(e/a)=k$ and $\rk(d/M_0e)\le \rk(ae)=\ell$.  It follows that $n\le k+\ell$.
In particular, $n:=\rk(de/a)$ is finite.   To show that $n\ge k+\ell$, it suffices to produce an $n$-chain for $\tp(de/a)$.
For this, since $\rk(d/ea)=\ell$, find $M_0\preceq \dots\preceq M_\ell$ with $ea\subseteq M_0$ and 
$\pcl(M_id)\cap(M_{i+1}\setminus M_i)\neq \emptyset$ for all $i<\ell$.

Next, fix an isomorphism $f:N\rightarrow M_0$ fixing $ea$ pointwise.
Since $\rk(e/a)=k$, there exist $M_{-k}\preceq \dots\preceq M_0$ with $a\subseteq M_{-k}$ and $\pcl(M_ie)\cap(M_{i+1}\setminus M_i)\neq \emptyset$ for all $-k\le i<0$.
The concatenation of these gives a $(k+\ell)$-chain $M_{-k}\preceq \dots M_\ell$ witnessing that $\rk(de/a)\ge k+\ell$.

(2)  By monotonicity, if $\rk(de/a)<\omega$, then both $\rk(d/ea)$ and $\rk(e/a)$ are finite as well (and, in fact, are at most $\rk(de/a)$).
So we are done by (1).
\end{proof}

We record the following immediate corollary (recall Definition~\ref{bfP}).

\begin{Corollary}  \label{add}   If $\At$ is finitely ranked 
then  $$\rk(ab/c)=\rk(a/bc)+\rk(b/c)$$
for all finite $a,b,c$ from $N$.
\end{Corollary}

Next, we consider infinite chains and see that their existence characterizes $\rk(d/a)=\infty$.
The chains defined here are crucial for the construction of $2^{\aleph_1}$ models in $\aleph_1$.

\begin{Definition} \label{chaindef} {\em  Fix $\tp(d/a)\in \P$.
A {\em $d/a$-chain} is an
$\omega$-sequence $\<(M_i,c_i):i\in\omega\>$  with union $M^*$ such that
\begin{enumerate}
\item  $a\subseteq M_0$ and $c_0$ is meaningless;
\item  $M_0\preceq M_1\preceq\dots M^*\preceq N$ is a nested sequence of (countable atomic) models;
\item  $c_{i+1}\in M_{i+1}\setminus M_i$ for every $i\in\omega$; and
\item  $c_{i+1}\in \pcl(M_id)$.
\end{enumerate}
A {\em better} $d/a$-chain also satisfies:
\begin{enumerate}
\setcounter{enumi}{4}
\item  For every $c\in\pcl(M^*d)$, if $\rk(d/M^*c)=\infty$, then $c\in M^*$; and
\item For every finite $e\in M^*$, every non-pseudoalgebraic 1-type $q\in S_{at}(e)$ is realized in $N\setminus M^*$.
\end{enumerate}
}
\end{Definition}

Note that it follows from Clauses (3) and (4) that $d\not\in\bigcup_{i\in\omega} M_i$.  
With Proposition~\ref{charinfty} we show that having a (better) $d/a$-chain characterizes $\rk(d/a)=\infty$.
We begin by defining finite approximations of a $d/a$-chain.

\begin{Definition} \label{approximate}
{\em  A $d/a$-approximation is a sequence $\ee=\<(e_i,c_i):i\in\lg(\ee)\>$ where
\begin{enumerate}
\item  $a\subseteq e_0\subseteq N$, $c_0=\emptyset$ and $\lg(\ee)<\omega$;
\item  $c_{i+1}\in\pcl(e_i,d)\setminus\pcl(e_i)$; and
\item $e_i\cup\{c_{i+1}\}\subseteq e_{i+1}$;
\item  $\rk(d/e_{\lg(\ee)})=\infty$.
\end{enumerate}
}
\end{Definition}

\begin{Lemma}  \label{finitestr}  Given any $n\ge 1$ and any $d/a$-approximation $\ee=\<(e_i,c_i):i\in\lg(\ee)\>$ of length $n=\lg(\ee)$, there is a sequence
 $M_0\preceq \dots \preceq M_{n-1}\preceq N$
such that for each $i<n$
\begin{itemize}
\item  $e_i\subseteq N_i$; and
\item  $c_{i+1}\not\in N_i$
\end{itemize}
\end{Lemma}

\begin{proof} By reverse induction.  First, since $c_n\not\in\pcl(e_{n-1})$, there is a (countable, atomic) model $N_{n-1}$ that contains $e_{n-1}$ but not $c_n$.
Next, we work inside $N_{n-1}$.  As $c_{n-1}\subseteq e_{n-1}$, $c_{n-1}\subseteq N_{n-1}$.  Since $c_{n-1}\not\in\pcl(e_{n-2})$, there is
$N_{n-2}$, which we can construct inside of $N_{n-1}$, that contains $e_{n-2}$
but not $c_{n-1}$.  Now continue.
\end{proof}

The following Proposition is a slight strengthening  of Theorem~\ref{summary}(1).  Better $d/a$-chains  constitute the Data used in proving Theorem~\ref{bigone}.

\begin{Proposition}  \label{charinfty}  The following are equivalent for any $\tp(d/a)\in\P$.
\begin{enumerate}
\item  
$\rk(d/a)=\infty$;
\item  A $d/a$-chain exists;
\item    A better $d/a$-chain exists.
\end{enumerate}
\end{Proposition}

Proof.  $(3)\Rightarrow(2)$ is trivial.

$(2)\Rightarrow(1)$:    
Fix  a $d/a$-chain $\<(M_i,c_i):i\in\omega\>$.  To show that $\rk(d/a)=\infty$, it suffices to prove that for all ordinals $\alpha$,

\begin{quotation}
For every finite $e\subseteq \bigcup M_i$, $\rk(d/e)\ge\alpha$.
\end{quotation}
We  establish this by induction on $\alpha$.  Fix $\alpha$ and assume that this holds for every $\beta<\alpha$.
Fix any finite $e\subseteq \bigcup M_i$.   We directly argue that $\rk(d/e)\ge\alpha$
from the definition of $\rk$.  So fix $r\in S_{at}(e)$ and $\beta<\alpha$.
Choose $n$ such that $e\subseteq M_n$.  Pick any realization $a'$ of $r$ in $M_n$.
Since $c_{n+1}\in\pcl(M_nd)$, we can find a finite $b\subseteq M_n$ such that
$c_{n+1}\in \pcl(ea'bd)$.  Let $q=\tp(d/ea'bc_{n+1})$.  By our inductive hypothesis, $\rk(q)\ge\beta$.  Thus, $\rk(d/e)\ge\alpha$ by the definition of $\rk$.

$(1)\rightarrow(3)$:  Suppose $\rk(d/a)=\infty$.  Trivially, $\<(a,\emptyset)\>$ is a $d/a$-approximation of length 1.  We will construct a $d/a$-chain satisfying 
Clause~(5) of  Definition~\ref{chaindef} in $\omega$ steps and  then modify it to obtain Clause~(6) as well.  To get the first part,  
it suffices to show that any $d/a$-approximation can be extended in each of three ways.

\medskip
\noindent{\bf Extending the sequence}

Fix any $d/a$-approximation $\ee$ of length $n$.  In order to get a $d/a$-approximation of length $(n+1)$ extending $\ee$, first note that $\rk(d/e_{n-1})=\infty$.  Thus, 
taking $r=\emptyset$ (or, if you prefer, let
$r(y):=`y=e_{n-1}'$) there are $b,c$ such that $c\not\in\pcl(e_{n-1}b)$, but
$c\in\pcl(e_{n-1}bd)$, and $\rk(d/e_{n-1}bc)=\infty$.  So, let $e_{n-1}'=e_{n-1}b$ and $e_n'=e_{n-1}bc$ (with $e_j'=e_j$ for all $j<n-1$).

\medskip
\noindent{\bf Enlarging $e_j$ one step toward a model}

Fix any $d/a$-approximation
$\ee=\<(e_i,c_i):i\in\lg(\ee)\>$ of length $n=\lg(\ee)$.
Choose any $j<n$ and fix any consistent formula $\phi(x,e_j)$.  We will produce a larger $d/a$-approximation $\ee'$ where $e_j'$ contains a realization of $\phi(x,e_j)$.  To do this, first choose a sequence of models $N_0\subseteq N_1\subseteq\dots\subseteq N_{n-1}$ as in Lemma~\ref{finitestr}.  So, $c_i\subseteq e_i\subseteq N_i$ and $c_{i+1}\in N_i$.  Next, choose $a^*$ from $N_j$ realizing $\phi(x,e_j)$.  Let $r^*=\tp(a^*/e_{n-1})$.

Finally, we apply Proposition~\ref{Prop2}(2) to get $a',b,c$ with $a'$ realizing $r^*$ and $\rk(d/e_{n-1}a'bc)=\infty$.  Let $q=\tp(d/e_{n-1}a')$ be the restriction, which also has $\rk(q)=\infty$.
Now, for $i<j$, let $e_i'=e_i$, while $e_i'=e_ia'$ for all $j\le i<n$.

\medskip
\noindent{\bf One step toward Clause (5) of Definition~\ref{chaindef}}

Fix any $d/a$-approximation
$\ee=\<(e_i,c_i):i\in\lg(\ee)\>$ of length $n=\lg(\ee)$,  and choose any $c\in\pcl(e_{n-1}d)$.  If $\rk(d/e_{n-1}c)<\infty$, then do nothing at this stage.
But if $\rk(d/e_{n-1}c)=\infty$, then affix $c$ to $e_{n-1}$ and continue.

\medskip

By dovetailing these three processes, we can construct a $d/a$-chain $\<(M_n,c_n):n\in\omega\>$ satisfying Clause~(5) in $\omega$
steps.  To obtain Clause~(6), Let $M^*=\bigcup\{M_n:n\in\omega\}$.  As $M^*$ is countable and $S_{at}(b)$ is countable for every finite tuple $b$ from $M^*$,
there is a countable,  atomic model $N'\succeq N$ such that for every finite $b\in M^*$ every non-pseudoalgebraic $q\in S_{at}(b)$ is realized in $N'\setminus M^*$.
As $\pcl(Z,N)=\pcl(Z,N')$ for all sets $Z\subseteq N$, we see that $\<M_n:n\in\omega\>$ and $M^*$ is also a $d/a$-chain satisfying (5) with respect to $N'$
as well as with respect to $N$.
Thus, if $f:N'\rightarrow N$ is any isomorphism fixing $da$ pointwise, then $\<f(M_n):n\in\omega\>$ and $f(M^*)$ are a better $d/a$-chain in $N$, as required.
\qed

\section{Pseudo-minimal types and finitely ranked classes}

Fix a (complete) pseudo-minimal type $p\in S_{at}(\emptyset)$ and assume $\theta(x)$ isolates $p$.
Then for any finite tuple $a$ from $N$, the relation of pseudo-closure over $a$,  is an exchange space on $\theta(N)$.
That is, $(\theta(N),\pcl_a)$, where $c\in\pcl_a(B)$ iff $c\in\pcl(Ba)$ satisfies the van der Waerden axioms.  
This implies there is a good notion of independence, i.e., for any $a$ and any $\pcl_a$-closed $C$, any two maximal  $\pcl_a$-independent subsets of $C$ have the same cardinality,
which we  dub the {\em dimension} of $C$ over $a$.
We note the following easy facts about independent tuples from $\theta(N)$.

\begin{Lemma}   \label{Indepfacts}   Suppose $\theta(x)$ is a complete, pseudo-minimal formula and $a\subseteq N$ is any finite tuple.
\begin{enumerate}
\item  If a finite tuple  $\cbar\subseteq \theta(N)$ is independent over $a$, then there is $M\preceq N$, with $a\subseteq M$, but $\cbar\cap M=\emptyset$.
\item  Say $\cbar=\cbar_1\cap \cbar_2$ is any partition, then $\cbar_1$ is independent over $\cbar_2 b$, hence there is a model $M\preceq N$
with $\cbar_2 a\subseteq M$ and $\cbar_1\cap M=\emptyset$.
\item  For any $\cbar\in \theta(N)^n$, $\cbar$ is independent over $a$ if and only if $\rk(\cbar/a)=n$.
\end{enumerate}
\end{Lemma}

The reader is cautioned, however, that although any two elements of $\theta(N)$ have the same 1-type over the empty set, there can be infinitely many 2-types of independent tuples 
in $\theta(N)^2$, e.g., if $\At$ is the class of atomic models of REF(bin), the theory of infinitely many refining equivalence relations, where each $E_{n+1}$ partitions each $E_n$-class into two pieces.   Thus, the apt analogue of $\theta(N)$ is that of a weakly minimal formula in the first order context.  Despite this, we have the following, which follows from the homogeneity of $N$.

\begin{Lemma}    \label{shuffle}  Suppose $a, b\subseteq N$ are finite and $\cbar\in\theta(N)^n$ is independent over $a$.  Then there is $\cbar'\in\theta(N)^n$
such that $\tp(\cbar/a)=\tp(\cbar'/a)$ and $\cbar'$ is independent over $ab$.
\end{Lemma}

\begin{proof}  By Lemma~\ref{Indepfacts}(1), choose $M\preceq N$ with $a\subseteq M$ and $\cbar\cap M=\emptyset$.  Choose an isomorphism
$f:N\rightarrow M$ with $f(a)=a$ and let $b'\in M$ be such that $\tp(ab)=\tp(ab')$.  Since $\cbar\cap M=\emptyset$, $\cbar$ is independent over $ab'$.
Choose an automorphism $\sigma\in Aut(N)$ with $\sigma(a)=a$ and $\sigma(b')=b$.  Then $\cbar':=\sigma(\cbar)$ is independent over $ab$ and
$\tp(\cbar'/a)=\tp(\cbar/a)$.
\end{proof}

For the remainder of this section, we assume that $\At$ is finitely ranked, so we have full additivity of rank.

\begin{Definition} \label{domdef} {\em  Suppose $b,h\subseteq N$ and $\cbar\subseteq \theta(N)$.  We say {\em  $\cbar$ $\theta$-dominates $b$ over $h$} if
\begin{enumerate}
\item  $\cbar$ is independent over $h$; and
\item  For all $\cbar^*\subseteq \theta(N)$ and all $h^*\supseteq h$,
if $\cbar\cbar^*$ is independent over $h^*$, then $\cbar^*$ is independent over $h^*\cbar b$.
\end{enumerate}
}
\end{Definition}

Under the assumption that $\At$ is finitely ranked, the following existence lemma shows that dominating sets are easily attained.

\begin{Proposition}  \label{domexistence}  ($\At$ finitely ranked).     Suppose  $b,d\subseteq N$ are finite and $\cbar\subseteq \theta(N)$ is independent over $d$.  Then:
\begin{enumerate}
\item   There is a finite $h\supseteq d$ such that $\cbar$ $\theta$-dominates $b$ over $h$.
\item  Moreover, if $\cbar^*\subseteq \theta(N)$ is initially chosen such that $\cbar^*\cbar$ is independent over $d$, then we may additionally
have $\cbar^*\cbar$ is independent over $h$.
\end{enumerate}
\end{Proposition}

\begin{proof}  (1) Among all finite tuples $h\subseteq N$ with $h\supseteq d$ and $\cbar$ independent over $h$, choose
one such that $\rk(b/h\cbar)$ is minimized.  We argue that $\cbar$ $\theta$-dominates $b$ over $h$.
To see this, choose any $\cbar'\subseteq \theta(N)$ and $h'\supseteq h$ such that $\cbar'\cbar$ is independent over $h'$.
We verify that $\cbar'$ is independent over $h'\cbar b$ by proving that $\rk(\cbar'/h'\cbar b)=\rk(\cbar'/h'\cbar)=\lg(\cbar')$.
The second equality is clear, as $\cbar'\cbar$ is independent over $h'$, but the first equality takes some work.

Note that $\cbar$ is independent over $h'\cbar'$, so the minimality property of $h$ yields
$$\rk(b/h'\cbar'\cbar)=\rk(b/h'\cbar)$$
Now we use additivity of $\rk$, Corollary~\ref{add},  twice.  On one hand,
$$\rk(b\cbar'/h'\cbar)=\rk(b/h'\cbar\cbar')+\rk(\cbar'/h'\cbar)$$
while on the other hand,
$$\rk(b\cbar'/h'\cbar)=\rk(\cbar'/h'\cbar b )+\rk(b/h'\cbar)$$
Combining these three equalities gives the requisite $\rk(\cbar'/h'\cbar b)=\rk(\cbar'/h'\cbar)$.

(2)  Now suppose $\cbar^*$ is given in advance with $\cbar^*\cbar$ independent over $d$.   By (1), choose $h\supseteq d$ with $\cbar$ $\theta$-dominating $b$ over $h$.
This $h$ might not have $\cbar^*\cbar$ independent over $h$, but we apply Lemma~\ref{shuffle} to get one that is.
Choose $\cbar''\subseteq \theta(N)$ such that $\tp(\cbar''/bd\cbar)=\tp(\cbar^*/bd\cbar)$ with $\cbar''$ independent over $bd\cbar h$.
Note that  since $\cbar$ is independent over $h$, so is $\cbar''\cbar$.  

As they have the same type, choose  an automorphism $\sigma\in Aut(N)$ such that $\sigma\mr{bd\cbar}=id$ and $\sigma(\cbar'')=\cbar^*$.
Put $h^*:=\sigma(h)$.   By the automorphism, $\cbar^*\cbar$ is independent over $h^*$.  As well,
 $h^*\supseteq d$ and since $\tp(\cbar b h)=\tp(\cbar b h^*)$, $\cbar$ $\theta$-dominates $b$ over $h^*$.
 \end{proof}

We obtain the following Corollary that may be helpful in constructing an atomic model of size continuum.

\begin{Corollary}  \label{crit}   ($\At$ finitely ranked)  Suppose $\cbar^*\subseteq\theta(N)$ is given, and  $\cbar, b, h, \cbar', b'$ satisfy
\begin{enumerate}
\item  $\cbar$ $\theta$-dominates $b$ over $h$;
\item  $\cbar^*\cbar$ is independent over $h$;
\item  $\cbar'$ is independent over $\cbar^*h\cbar b$; and
\item $\tp(\cbar' b'/h)=\tp(\cbar b/h)$.
\end{enumerate}
Then $\cbar^*$ is independent over $h \cbar b \cbar' b'$.
\end{Corollary}

\begin{proof}  By (3), $\cbar'$ is independent over $\cbar^* \cbar h$,
hence by (2), 
\begin{quotation}  $\cbar^*\cbar'\cbar$ is independent over $h$
\end{quotation}
so by (1), $\cbar^*\cbar'$ is independent over $h \cbar b$.

As they have the same type, (4) implies that $\cbar'$ $\theta$-dominates $b'$ over $h$.  
Apply this with $h^*:=h\cbar b$ givens $\cbar^*$ independent over $h \cbar b \cbar' b'$.
\end{proof}

\section{Few atomic models implies  ranked}  \label{3}
 \numberwithin{Theorem}{subsection}

Throughout this section, we assume that $\At$ is not ranked, i.e., $\rk(p)=\infty$ for some $p\in\P$.  The goal of this section
is to prove that $\At$ contains a family of $2^{\aleph_1}$ non-isomorphic (atomic) models, each of size $\aleph_1$.
To begin, we fix the following Data that is obtained from the existence of a better $a^*/a$ chain via Proposition~\ref{charinfty}.
(Note that as we are producing so many models, we may absorb $a$ into the signature so we are left with a better $a^*/\emptyset$ chain.)
Similarly, we can always add a constant symbol to the language, thereby assuring that $\pcl(\emptyset)\neq\emptyset$.

\begin{Data}  \label{DataA}
 Fix a countable $N^*\in\At$, an elementary chain $\<M_n:n\in\omega\>$ with union $M^*\preceq N^*$, a
distinguished element $a^*$, elements $c_m\in M_{m+1}\setminus M_m$ and finite tuples $\dbar_m$ from $M_m$ such that
\begin{enumerate}
\item  Each $c_m\in\pcl(a^*\dbar_m,N^*)$;
\item  For every $e\in N^*\setminus M^*$,  if $e\in\pcl(M^*a^*,N^*)$, then $\rk(a^*/M^*e)<\infty$; and
\item  For every finite $b\in M^*$, every non-pseudoalgebraic $q\in S_{at}(b)$ is realized in $N^*\setminus M^*$.
\end{enumerate}
\end{Data}

Our strategy is similar to the proof of the non-structural result in \cite{BLSmanymod}.
In Subsection~\ref{order}, which is nearly identical with \cite[\S 4.1]{BLSmanymod}, we define a family of orders $I^S$, indexed by stationary/costationary subsets of $\omega_1$
and discuss weakly striated models indexed by such an $I^S$.
The forcing $(\Q_I,\le)$ is defined in Subsection~\ref{defineforcing}.  In Subsection~\ref{proof1} we prove Proposition~\ref{forcing}, 
that verifies that  $(\Q_I,\le)$ forces the existence of  atomic $N_I$ with cardinality $\aleph_1$ determined by the order $I$. 
Finally, in Subsection~\ref{final}, we show that having this uniform process of forcing an extension allows us build families of atomic models
$(N_S:S\subseteq\omega_1)$ {\em in ${\mathbb V}$} (as opposed to in a forcing extension) in such a manner that if $S\triangle S'$ is stationary,
then $N_S\not\cong N_{S'}$.  Theorem~\ref{summary}(2) follows easily from this.

\subsection{A class of linear orders and weakly striated models}  \label{order}

We begin by describing a class of $\aleph_1$-like linear orders, colored by a unary predicate $P$ and an equivalence relation $E$ with convex classes.
A related notion of striation was   discussed in \cite{BLSmanymod}.

\begin{Definition} {\em  Let $\L=\{<,P,E\}$ and let $\II^*$ denote all $\L$-structures $(I,<,P,E)$ satisfying:
\begin{enumerate}
\item $(I,<)$ is an $\aleph_1$-like dense linear order  (i.e., $|I|=\aleph_1$, but $pred_I(a)$ is countable for every $a\in I$) with minimum element $\min(I)$;
\item $P$ is a unary predicate;
\item  $E$ is an equivalence relation on $I$ with convex classes such that
\begin{enumerate}
\item  If $t=\min(I)$ or if $P(t)$ holds, then $t/E=\{t\}$;
\item  Otherwise, $t/E$ is (countable) dense linear order without endpoints.
\end{enumerate}
\item  The condensation $I/E$ is a dense linear order with minimum element, no maximum element, such that both
sets $\{t/E:P(t)\}$ and $\{t/E:\neg P(t)\}$ are dense in it.
\end{enumerate}
}
\end{Definition}

We are interested in well-behaved proper initial segments of elements of $\II^*$.

\begin{Definition}  {\em  Fix $(I,<,P,E)\in\II^*$.  A proper initial segment $J\subseteq I$ is
{\em endless} if it has no maximum element and is {\em suitable} if, for every $s\in J$ there is $t\in J$, $t>s$,
with $\neg E(s,t)$.  Finally, call a suitable $J$ {\em seamless} if $I\setminus J$ has no minimal $E$-class.
}
\end{Definition}

Clearly, $J$ suitable implies $J$ endless.
The following  Lemma and Construction are  Lemma~4.1.3  and Construction~4.1.4 of   \cite{BLSmanymod}.

\begin{Lemma} \label{seamless}
 If $(I,<,P,E)\in\II^*$ and $J\subseteq I$ is a seamless proper initial segment, then for every finite $\SS\subseteq I$ and $w\in J$ such that $w>s$ for every $s\in\SS\cap J$, there is an automorphism
$\pi$ of $(I,<,P,E)$ that fixes $\SS$ pointwise, and $\pi(w)\not\in J$.
\end{Lemma}

\begin{Construction}\label{IS}  Let $S\subseteq\omega_1$ with $0\not\in S$.  There is
$I^S=(I^S,<,P,E)\in\II^*$ that has a continuous, increasing sequence
$\<J_\alpha:\alpha\in\omega_1\>$ of proper initial segments such that:
\begin{enumerate}
\item If $\alpha\in S$, then $I^S\setminus J_\alpha$ has a minimum element $a_\alpha$ satisfying $P(a_\alpha)$; and
\item  If $\alpha\not\in S$ and $\alpha>0$, then $J_\alpha$ is seamless.
\end{enumerate}
\end{Construction}

\begin{Definition}  {\em Fix an atomic  $N\in \At$ and some $I=(I,<,E,P)\in\II^*$.
We say $N$ is {\em  weakly striated by $I$} if there are $\omega$-sequences $\<\abar_t:t\in I\>$ satisfying:
\begin{itemize}
\item  $N=\bigcup\{\abar_t:t\in I\}$;  (As notation, for $t\in I$, $N_{<t}=\bigcup\{\abar_j:j<t\}$.)
\item  If $t=\min(I)$, then $\abar_t\subseteq\pcl(\emptyset, N)$;
\item  For $t>\min(I)$, $a_{t,0}\not\in\pcl(N_{<t},N)$;
\item  For each $s$ such that $\neg P(s)$ and  for every $n\in\omega$, $a_{s,n}\in\pcl(N_{<s}\cup\{a_{s,0}\},N)$.
\end{itemize}
}

\end{Definition}

The final clause of this definition is weaker than the notion of {\em striation} \cite[Definition 4.5]{BLSmanymod} in that it only constrains levels where $P$ fails.  However, the definition of the forcing
will put a constraint (albeit weaker) on levels for which $P$ holds.

\medskip

Note:  In the definition above, we allow $a_{s,m}=a_{t,n}$ in some cases when $(s,m)\neq(t,n)$.  However, if $s<t$, then
the element $a_{t,0}\neq a_{s,m}$ for any $m$.

The idea of our forcing will be to force the existence of a  weakly striated atomic model $N_I$ indexed by the linear order $I\in\II^*$ formed from Construction~\ref{IS}.
We will begin with the array $X_I=\{x_{t,n}:t\in I,n\in\omega\}$ of symbols.  The forcing will give us a complete type $\Gamma$ in the variables $X_I$ in which every finite
subset realizes a principal type with respect to $T$.  This $\Gamma$ defines a congruence on $X_I$, with $x_{t,n}\sim x_{s,m}$ if
$\Gamma\vdash x_{t,n}=x_{s,m}$.  The universe of $N_I$ will be $X_I/\sim$, with $\Gamma$ providing interpretations of each symbol of $\tau$.
Such an $N_I$ will have a `built in' continuous sequence
$\<N_\alpha:\alpha\in\omega_1\>$ of countable, elementary substructures, where
the universe of $N_\alpha$ will be $\{[x_{t,n}]:t\in J_\alpha, n\in\omega\}$ for some suitable initial  segment $J_\alpha$ of $I$.  The idea will be to use the data concerning the
pair of models $(M^*,N^*)$ given in Data~\ref{DataA}
to make $$\{\alpha\in\omega_1:I\setminus J_\alpha\ \hbox{has a minimum element\}}$$
(infinitarily) definable in a language $\tau^*$ that codes all of the data mentioned above.

\subsection{The forcing}  \label{defineforcing}

Fix the Data from Data~\ref{DataA}.
Fix a stationary/costationary subset $S\subseteq\omega_1$ and use Construction~\ref{IS} to form  $I^S=(I,<,E,P)\in\II^*$.
We describe three adjectives on a weakly striated model.

\begin{Definition}  {\em  Suppose $N$ is  weakly striated by $(I,<,P,E)$, $J\subseteq I$ suitable, and $b\in N\setminus N_J$.
\begin{itemize}
\item  {\em $b$ $\rk\ \infty$-catches $N_J$} if, $\rk(b/N_J,N)=\infty$ and for every $e\in N$, $e\in\pcl(N_{J}\cup\{b\},N)\setminus N_{J}$ implies $\rk(b/N_J\cup\{e\},N)<\infty$.
\item {\em $b$ admits a cofinal chain in $N_J$} if there is a strictly increasing, cofinal  sequence $\<s_n:n\in\omega\>$ from  $J$ such that for every $n$,
$\pcl(N_{s_n}\cup\{b\},N)\cap N_{s_{n+1}}\neq N_{s_n}$.

\item  {\em $b$ has bounded effect in $N_J$} if there exists $s^*\in J$ such that $\pcl(N_{s}\cup\{b\},N)\cap N_J=N_{s}$ for every $s>s^*$ with $s\in J$.
\end{itemize}
}
\end{Definition}

Note that  any sequence $\<s_n:n\in\omega\>$ witnessing `$b$ admits a cofinal chain' must also satisfy $s_n/E<s_{n+1}/E$.  As well, any infinite subsequence would also
be a witness.
Clearly, if $b$ has bounded effect in $N_J$, then $b$ cannot admit a cofinal chain in $N_J$.

\begin{Proposition}  \label{forcing}  Suppose $N^*,a^*, M^*=\bigcup_{m\in\omega} M_m$, $\dbar_m, c_m$ are from Data~\ref{DataA}
and let $I=I^S$ be from Construction~\ref{IS} with $S$ stationary/costationary.
There is a c.c.c.\ forcing $\QQ_I$ such that in $V[G]$, there is a full, atomic $N_I\models T$
weakly striated by $(I,<)$ such that:
\begin{enumerate}
\item  For every suitable initial segment $J\subseteq I$, $N_J\preceq N_I$;
\item  If  $t\in I$ and $P(t)$ holds,
then $a_{t,0}$ $\rk\ \infty$-catches and admits a cofinal chain in $N_t$; and
\item  If $J\subseteq I$ is seamless, then for every $b\in N_I\setminus N_{J}$,
if $b$ $\rk\ \infty$-catches $N_J$, then $b$ has bounded effect in $N_J$.
\end{enumerate}
\end{Proposition}

Recall that by naming a constant if necessary, we are assuming $\pcl(\emptyset,M)\neq\emptyset$.  
Fix  a specific complete formula $\gamma(y)$ isolating a specific type of an element in $\pcl(\emptyset,M)$.
Additionally, fix, for the whole of the proof, some $(I,<,E,P)\in\II^*$.  We aim to construct an atomic model $N_I\in\At$, whose complete diagram consists of
$\{x_{t,n}:t\in I,n\in\omega\}$ that is  weakly striated by $(I,<)$.  
 We first accomplish this via the forcing notion $(\QQ_I,\le_{\QQ})$, defined below.  
Elements of  $\QQ_I$ will record `finite approximations' of this complete diagram.
More precisely:

\begin{Definition} \label{finiteseq} {\em
An {\em approximation  sequence $\xbar$ from $\<x_{t,n}:t\in I, n\in\omega\>$} has the form $\xbar=\<x_{t,m}:t\in u, m<n_t\>$, where
$u\subseteq I$ is finite and $1\le n_t<\omega$ for every $t\in u$.  Given a finite sequence $\xbar$ indexed by $u$ and $\<n_t:t\in u\>$
and given a proper initial segment $J\subseteq I$, let $u\mr J=u\cap J$ and $\xbar\mr J=\<x_{t,m}:t\in u\mr J, m<n_t\>$.  As well, if $p(\xbar)$
is a complete type in the variables $\xbar$, then $p\mr J$ denotes the restriction of $p$ to $\xbar\mr J$, which is necessarily a complete type.
For $s\in I$, the  symbols $u\mr{<s}$ and $\xbar\mr {\le s}$ are defined analogously, setting $J=I\mr {<s}$ and $I\mr {\le s}$, respectively.
}
\end{Definition}

\medskip

\par\noindent {\bf Definition of ${\bf (\QQ_{I},\le_\QQ)}$:}  \quad
 $p\in \QQ_I$ if and only if the following conditions hold:
 \begin{enumerate}

\item $p$ is a complete (principal) type with respect to $T$ in the variables $\xbar_p$, which are a finite sequence indexed by $u_p  \subset  I$ and $n_{p,t}\in\omega$.
In addition, $p$ comes equipped with a pair of functions $\gbar_p=(g_{0,p},g_{1,p})$, each of which have domain $u_p\cap P$. .
(When $p$ is understood we sometimes write $n_t, g_0,g_1$, etc.);

\item Striation constraints:
\begin{enumerate}

\item   $x_{\min(I),0}\in\xbar_p$ and $\gamma(x_{\min(I),0})\in p(\xbar_p)$;
\item If $t=\min(I)$, then $p$ `says' $\{x_{t,n}:n<n_t\}\subseteq\pcl(\emptyset)$;
\item  For all $t\in u_p$, $t\neq\min(I)$, $p$ `says' $x_{t,0}\not\in\pcl(\xbar_p\mr {<t})$; and
\item  For all $s\in u_p$ such that $\neg P(s)$ and $m<n_s$, $p$ `says' $x_{s,m}\in\pcl(\xbar_p\mr {<s}\cup\{x_{s,0}\})$.
\end{enumerate}

\item  For each $t\in u_p\cap P$, $g_{0}(t)$ is a finite approximation to a cofinal sequence below $t$:
\begin{enumerate}
\item  $\dom(g_{0}(t))$ is a positive integer $\ell_{p,t}$ and for $i<\ell_{p,t}$ we write $s_i$ for $g_0(t)(i)$;
\item  Every $s_i\in u_p$ and $\min(I)<s_0/E<s_1/E<\dots<s_{\ell_{p,t}-1}/E<t$;
\item  $s_{\ell_{p,t-1}}$ is in the topmost $E$-class of $u_p$ below $t$; and
\item  $\neg P(s_i)$
\end{enumerate}

\item  For each $t\in u_p\cap P$, $g_1(t)$ gives an `elementary map' from $\xbar_p\mr{\le t}$ into $N^*$.  For $X\subseteq \xbar_p\mr{\le t}$,
let $g_1(t)[X]:=\{g_1(t)(x):x\in X\}$, the image of $g_1(t)\mr X$.  We require that each $g_1(t)$ satisfy:
\begin{enumerate}
\item  For any subset $\wbar\subseteq \xbar_p\mr{\le t}$, and $\tau$-formula $\phi(\wbar)$, $p(\xbar_p)\vdash\phi(\wbar)$ if and only if
$N^*\models\phi(g_1(t)(\wbar))$;
\item  $g_1(t)(x_{t,0})=a^*$;
\item  $g_1(t)[\xbar_p\mr{=t}]\subseteq N^*\setminus M^*$; and
\item  $g_1(t)[\xbar_p\mr{<t}]\subseteq M^*$
\end{enumerate}

\item  Interconnections:  For every $t\in u_p\cap P$ and for each $i<\ell_{p,t-1}$, there is some $m\in\omega$ such that, letting $s_i=g_{0}(t)(i)$ and recalling
Data~\ref{DataA}:
\begin{enumerate}
\item  $\dbar_m\subseteq g_1(t)[\xbar_p\mr{\le s_i/E}]$;
\item  $g_1(t)(x_{s_{i+1},0})=c_m$
\end{enumerate}
\end{enumerate}

For $p,q\in\QQ_I$, we define
$p\le_{\QQ_I} q$ if and only if $\xbar_p\subseteq\xbar_q$, the complete type $p(\xbar_p)$ is the restriction of $q(\xbar_q)$ to $\xbar_p$, and for all
$t\in u_p\cap P$, $g_{0,q}$ end extends $g_{0,p}$, and $g_{1,q}$ extends $g_{1,p}$.

We make the following observations:
\begin{itemize}
\item  Striation constraint 2(a) implies that $\min(I)\in u_p$ for every $p\in\QQ_I$;
\item  For any $p\in\QQ_I$, for each $t\in u_p\cap P$, because of 3(c),  $\neg P(\max(u_p\cap I\mr{<t}))$;
\item  Because of the interplay between 4(c,d), for any $p\in \QQ_I$, for any $x_{t,n},x_{s,m}\in\xbar_p$, if $P(t)$ holds and $s<t$,
then $p\vdash x_{t,n}\neq x_{s,m}$; and
\item  If $t,t'\in u_p\cap P$ with $t'\neq t$, then, other than each of $g_1(t)$ and $g_1(t')$ being elementary maps on their
common domains, there is no assumption of coherence between these maps.  In particular, if $t<t'$, we do not enforce
that $g_1(t)[\xbar_p\mr{<t}]$ be contained in $g_1(t')[\xbar_p\mr{<t'}]$.
\end{itemize}

\subsection{Proof of Proposition~\ref{forcing}}  \label{proof1}

\medskip

We begin this rather lengthy subsection by describing ways of extending conditions,
The first two Lemmas are immediate.

\begin{Lemma}  \label{truncate}
For every $p\in\QQ_I$ and suitable $J\subseteq I$, $p\mr J\in\QQ_I$ and $p\mr J\le_{\QQ_I} p$.
\end{Lemma}

\begin{Lemma}  \label{extendauto}
Every automorphism $\pi$ of $(I,<,E,P)$ naturally extends to an automorphism $\pi'$ of $\QQ_I$
via the mapping $x_{t,n}\mapsto x_{\pi(t),n}$.
\end{Lemma}

Our aim is to prove Proposition~\ref{compatible} below, which yields that the forcing is ccc.  This will require a number of preparatory lemmas.

\begin{Definition}  \label{trivialcons}  {\em  A condition $p\in\QQ_I$ is {\em non-trivial} if $u_p\neq\{\min(I)\}$. [Recall that $\min(I)\in u_p$ by
Striation constraint 1(a).]}
\end{Definition}

\begin{Definition}  \label{simple} {\em For $p,q\in\QQ_I$, we say {\em $q$ is a simple extension of $p$} if $q\ge_{\QQ} p$ and
`$u_q$ has no new $E$-classes' i.e., for every $s\in u_q$ there is $s'\in u_p$ such that $E(s,s')$.
}
\end{Definition}

What makes simple extensions simple is that (among other things) $g_0$ cannot be increased.

\begin{Lemma} \label{simplelemma} If $q$ is a simple extension of $p$, then
\begin{enumerate}
\item  $u_q\cap P=u_p\cap P$; and
\item $g_{0,q}=g_{0,p}$.
\end{enumerate}
\end {Lemma}

Proof.  (1) is immediate, since $P(t)$ implies that $t/E=\{t\}$.
(2)  First, because of (1), $g_{0,q}$ and $g_{0,p}$ have the same domain.
However, for any $t\in u_p\cap P$, the definition of $q\ge_\QQ p$ implies that $g_{0,q}(t)$ end extends $g_{0,p}(t)$.
In addition, the `last element' $g_{0,q}(t)(\ell_{q,t})$ is an element of the largest $E$-class below $t$.
But, by simplicity, the largest $E$-class of $p$ below $t$ is equal to the largest $E$-class below $t$.  As Condition 3(b) implies
that the $s_i/E$ are strictly increasing, we must have $g_{0,q}(t)=g_{0,p}(t)$.
$\qed_{\ref{simplelemma}}$

\begin{Definition}  \label{partitioned}  {\em  Suppose that  $\phi(y,\zbar)$ is a complete formula with respect to $T$ and $p\in\QQ_I$.
We say $\phi(y,\zbar)$ {\em includes $p$} if $\zbar=\xbar_p$ and $\phi(y,\zbar)\vdash\ p(\xbar_p)$.  For each $w\in u_p$,
define $\phi_w(y,\xbar_p\mr{\le w})$ to be the restriction of $\phi$ to the displayed variables.
  Clearly, each such $\phi_w(y,\xbar_p\mr{\le w})$
is a complete formula that includes $p\mr{\le w}$.

As $p(\xbar_p)$ describes a complete type and the relation `$a\in\pcl(\bbar,M)$' only depends on the complete type $\tp(a\bbar,M)$,
we say that a complete formula $\phi(y,\xbar_p)$ that includes $p$ is {\em pseudo-algebraic} if for some/every $M\in\At$ and for some/every
$\bbar$ from $M$ realizing $p(\xbar_p)$, $\phi(y,\bbar)$ is pseudo-algebraic in $M$.
}
\end{Definition}

\begin{Remark}  \label{useremark}  {\em  Note that if $\phi(y,\xbar_p)$ is not pseudo-algebraic and $\theta(\zbar,\xbar_p)$ is any complete formula, then there is a
complete, non-pseudo-algebraic $\psi(y,\zbar,\xbar)$ extending $\phi(y,\xbar_p)\wedge\theta(\zbar,\xbar_p)$.  To see this, choose any
$c\abar$ realizing $\phi(y,\xbar_p)$ in $N$.  As $\phi(y,\abar)$ is not pseudo-algebraic, choose $M\preceq N$ containing $\abar$ but not $c$.
Choose any $\ebar$ from $M$ witnessing $\theta(\ebar,\abar)$.  Then the complete formula $\psi$ isolating $\tp(c,\ebar,\abar)$ in $N$ suffices.
}
\end{Remark}

\begin{Lemma}  \label{bigHenkin}
Suppose $p\in\QQ_I$  is non-trivial and $\phi(y,\xbar_p)$ includes $p$.  Then there is a one-point, simple extension $q$ of $p$ such that:
\begin{enumerate}
\item  $\xbar_q=\{x_{s,m}\}\cup\xbar_p$ for some $(s,m)\in I\times\omega$;
\item  $q(\xbar_q)=\phi(x_{s,m},\xbar_p)$;
\item  for every $w\in u_p$, if $\phi_w(y,\xbar_p\mr{\le w})$ is pseudo-algebraic, then $s\le w$; and
\item  if $\phi(y,\xbar_p)$ is not pseudo-algebraic and $u_p\subseteq J$ for some endless $J$, then  we can require $s\in J$ as well.
\end{enumerate}
Moreover, given any sequence $\<(t,c_t):t\in u_p\cap P, c_t\in N^*\>$ such that for each $t$, $\tp(c_t,g_{1,p}(t)[\xbar_p\mr{\le t}],N^*)$ contains
$\phi_t(y,\xbar_p\mr{\le t})$, we can find $q$ as above with the additional property that $g_{1,q}(x_{s,m})=c_t$ for all $t\ge s$.
\end{Lemma}

Proof.  Our proof will split into several cases.  However,  in all cases, as we are requiring $q$ to be a simple extension, we must have  $u_q\cap P=u_p\cap P$ and
$g_{0,q}=g_{0,p}$. Consequently, since $p\in\QQ_I$, Constraint groups~3~and~5 will be automatically satisfied for $q$.

As $p$ is non-trivial, $\max(u_p)\neq\min(I)$.

\medskip\par\noindent{\bf Case 1a:}  $\phi(y,\xbar_p)$ is not pseudo-algebraic and $\neg P(\max(u_p))$.

\medskip  Let $s'=\max(u_p)$.  Given any endless $J$ with $u_p\subseteq J$, as $s'/E$ is dense we can choose
$s\in s'/E$ such that $s\in J$, but $s>w$ for all $w\in u_p$.  Take $m=0$.  Let $\xbar_q=\{x_{s,0}\}\cup\xbar_p$ and let $q(\xbar_q)=\phi(x_{s,0},\xbar_p)$.
Take $\gbar_q=\gbar_p$ and there is really nothing to check.

\medskip\par\noindent{\bf Case 1b:}  $\phi(y,\xbar_p)$ is not pseudo-algebraic and $P(\max(u_p))$ holds.

\medskip  Let $t^*=\max(u_p)$.
We now consider two Subcases, depending on our choice of sequence $\<(t,c_t):t\in u_p\cap P\>$ in the `Moreover clause.'

\medskip\par\noindent{\bf Subcase:}  $c_{t^*}\in N^*\setminus M^*$.  Here, let $(s,m)=(t^*,n_{t^*})$.  Put $\xbar_q=\{x_{t^*,n_{t^*}}\}\cup\xbar_p$ and $q(\xbar_q)=\phi$.
For $t\in u_p\cap P$ with $t<t^*$, put $g_{1,q}(t)=g_{1,p}(t)$ and there is nothing to check.  Finally, let $g_{1,q}(t^*)$ be the one-element extension of
$g_{1,p}(t^*)$ formed by putting $g_{1,q}(t^*)(n_{t^*})=c_{t^*}$.  As we are adding a new point at a level where $P$ holds, the Striation constraints are trivially satisfied,
and everything is easy.

\medskip\par\noindent{\bf Subcase:}  $c_{t^*}\in M^*$.  Here, Constraint 4(c) forbids us from putting the new element at level $t^*$.
Let $s'=g_{0,p}(t^*)(\ell_{t^*}-1)$.  Then $\neg P(s')$, and $s'/E$ is the maximal $E$-class represented in $u_p$ below $t^*$.
Choose $s\in I$ such that $E(s,s')$ holds, but $s>w$ for every $w\in u_p\setminus\{t^*\}$.  Take $m=0$.
That is, $\xbar_q=\{x_{s,0}\}\cup\xbar_p$ and put $q(\xbar_q)=\phi$.  As above, put $g_{1,q}(t)=g_{1,p}(t)$ for every $t\in u_p\cap P$ with $t<t^*$.
Finally, let $g_{1,q}(t^*)$ be the one-point extension of $g_{1,p}(t^*)$ formed by putting $g_{1,q}(t^*)(x_{s,0})=c_t$.

The non-trivial point to check is that this extension $q$ preserves the Striation constraints.
To see this, for $w\in u_p\setminus\{t^*\}$ there is nothing to check.  As $\phi(y,\xbar_p)$ is non-pseudo-algebraic, $x_{s,0}\not\in \pcl(\xbar_q\mr{<s})$.
And finally, as $g_{1,p}(t^*)(x_{t^*,0})=a^*\not\in M^*$, while $\{c_t\}\cup g_{1,p}(t^*)[\xbar_p\mr{<t^*}]\subseteq M^*$, we conclude that
$x_{t^*,0}\not\in\pcl(\xbar_q\mr{<t^*})$.

\medskip  {\em For the remainder of the proof, assume that $\phi(y,\xbar)$ is pseudo-algebraic.}  Indeed, let $w^*\in u_p$ be least such that  the restriction
$\phi_{w^*}(y,\xbar_p\mr{\le w^*})$ is pseudo-algebraic.  In each of the cases below, we will either take $s=w^*$, or $s$ will be less than $w^*$ but greater
than any $w\in u_p$ with $w<w^*$.  Thus, the Striation constraints for each $w\in u_q$ with $w<s$ will be trivially satisfied since they hold for $p$.
Also, since $\phi_{w^*}(y,\xbar_p\mr{\le w^*})$ is pseudo-algebraic, for every $w\in u_q$ with $w>w^*$ we will have $\pcl(\xbar_q\mr{\le w})=\pcl(\xbar_p\mr{\le w})$.
It follows from this that the Striation constraints for $q$ are satisfied for every $w\in u_q$ with $w>w^*$.  Because of this, in each of the cases below,
we only need to establish the Striation constraints for $q$ at levels $s$ and $w^*$.

\medskip\par\noindent{\bf Case 2a:}   $\neg P(w^*)$.

\medskip  In this case, let $(s,m)=(w^*,n_{p,w^*})$.  Take $\xbar_q=\{x_{s,m}\}\cup\xbar_p$ and $q(\xbar_q)=\phi$.  Note that since $\phi_{w^*}(y,\xbar_p\mr{\le w^*})$ is
pseudo-algebraic, the Striation constraints are satisfied for $q$ at level $w^*$, hence at all levels by the comments above.
To complete the description of $q$, for $t<w^*$, put $g_{1,q}(t)=g_{1,p}(t)$.  As well, given any sequence $\<(t,c_t):t \in u_p\cap P\>$ satisfying the `Moreover clause,'
for each $t>w^*$, let $g_{1,q}(t)$ be the one-point extension of $g_{1,p}(t)$ formed by putting $g_{1,q}(t)=c_t$.
Constraint group~4 is trivially satisfied for $q$ for each $t<w^*$.  So fix $t>w^*$.  By Constraint 4(d) on $p$, we have that $A=g_{1,p}(t)[\xbar_p\mr{<t}]\subseteq M^*$.
However,  since $t>w^*$, we also know that the restriction of $\phi$ to the variables $(y,\xbar_p\mr{<t})$ is pseudo-algebraic.  Thus, it follows that $c_t\in M^*$.
So Constraint~4(d) holds for $g_{1,q}(t)$ as well.  The other conditions are trivially satisfied, so $q\in\QQ_I$.

\medskip\par\noindent{\bf Case 2b:}   $P(w^*)$ holds and, letting $t^*=w^*$, $c_{t^*}\in N^*\setminus M^*$.

\medskip  In this case, we can place the new element at level $t^*$.  That is, $\xbar_q=\{x_{t^*,m}\}\cup\xbar_p$, where $m=n_{p,t^*}$ and $q(\xbar_q)=\phi$.
As $P(t^*)$ holds, the Striation constraints  at level $t^*$ hold for $q$ as they held for $p$.  As noted above, this implies that the Striation constraints hold for $q$ at all
levels.  Now, for $t<t^*$, let $g_{1,q}(t)=g_{1,p}(t)$ and for $t\ge t^*$, let $g_{1,q}(t)$ be the one-point extension of $g_{1,p}(t)$ formed by putting $g_{1,q}(t)(x_{t^*,m})=c_t$.
We must verify that Constraint group~4 is satisfied.  For $t<t^*$, this is trivial.  At level $t^*$, there is no problem as $c_{t^*}\in N^*\setminus M^*$.
For levels $t>t^*$, we argue just as in Case~2a) above.  That is, since $\phi_{t^*}(y,\xbar_p\mr{\le t^*})$ is pseudo-algebraic, we must have that $c_{t^*}\in M^*$.
Thus, there is no problem.

\medskip\par\noindent{\bf Case 2c:}   $P(w^*)$ holds and, letting $t^*=w^*$, $c_{t^*}\in M^*$.

\medskip  As in the second Subcase above, Condition 4(c) forbids us from adding the new element at  level $t^*$.  Let $s'\in u_p$ be maximal below $t^*$.
As $\neg P(s')$ holds, $s'/E$ is dense linear order, so we can choose $s\in s'/E$ with $s>s'$.  Let $\xbar_q=\{x_{s,0}\}\cup\xbar_p$ and $q(\xbar_q)=\phi$.
We need to check the Striation constraints at levels $s$ and $w^*$.  At level $s$, note that the minimality of $w^*=t^*$ implies that $x_{s,0}\not\in \pcl(\xbar_q\mr{<s})$, so we
are fine at level $s$.  At level $w^*=t^*$, note that $A=g_{1,p}(t^*)[\xbar_p\mr{<t^*}]\subseteq M^*$ and we are assuming $c_{t^*}\in M^*$.  As well, $g_{1,p}(t^*)(x_{t^*,0})=a^*$,
which is in $N^*\setminus M^*$.  Thus, the elementarily of the map assumed by the `Moreover clause' implies that $x_{t^*,0}\not\in\pcl(\xbar_q\mr{<t^*})$, so
the Striation constraints are satisfied at level $t^*$ as well.

Finally, we complete the description of $q$ by putting $g_{1,q}(t)=g_{1,p}(t)$ for all $t<t^*$, and for $t\ge t^*$, let $g_{1,q}(t)$ be the one-point extension of $g_{1,p}(t)$ given
by $g_{1,q}(t)(x_{s,0})=c_t$.  We need to show that Constraint group~4 is maintained.  This is trivial for all $t<t^*$.  At level $t^*$, it is satisfied because of our
assumption that $c_{t^*}\in M^*$.  Finally, fix any $t>t^*$ and recall that $A=g_{1,p}(t)[\xbar_p\mr{<t}]\subseteq M^*$.  As $x_{s,0}\in\pcl(\xbar_p\mr{<t})$, this means
that  $c_t\in\pcl(A,N^*)$, so $c_t\in M^*$, as required in 4(d).
$\qed_{\ref{bigHenkin}}$

Although Lemma~\ref{bigHenkin} is very strong, we still need some technique for producing extensions that need not be simple.  [Indeed, if $p$ is trivial, then
Lemma~\ref{bigHenkin} does not apply to $p$ at all.]

\begin{Lemma}  \label{alternate}
Suppose $p\in \QQ_I$ and $s\in I$ is chosen so that $s'<s$  for every $s'\in u_p$ and $\neg P(s)$.  Let $\phi(y,\xbar_p)$ be any complete formula including $p$
that is not pseudo-algebraic.
There is $q\in\QQ_I$ with $q\ge p$, $\xbar_q=\{x_{s,0}\}\cup\xbar_p$, and $q(\xbar_q)=\phi$.  \end{Lemma}

Proof. Simply put $\xbar_q=\{x_{s,0}\}\cup\xbar_p$ and $q(\xbar_q)=\phi$.  As $u_q=u_p\cup \{s\}$ and as $q\mr{\le t}=p\mr{\le t}$
for every $t\in u_p\cap P$, we can put $\gbar_q=\gbar_p$.  That $q$ satisfies the Striation constraints is immediate as $\phi$ is not pseudo-algebraic.
Thus, $q\in\QQ_I$ and $q\ge p$ as required.
$\qed_{\ref{alternate}}$

\begin{Definition}  \label{s-toppedD}
{\em  A non-trivial $p\in\QQ_I$ is {\em s-topped} if  $\neg P(\max(u_p))$.
}
\end{Definition}

Note that for any $p\in\QQ_I$ and any $t\in u_p\cap P$,  Clauses 3(c),(d) in the definition of the forcing imply that $p\mr{<t}$ is s-topped.

\begin{Lemma}  \label{s-toppedL}
Given any $p\in\QQ_I$ and any $s'\in I$ such that $u_p\subseteq I\mr{<s'}$, there is an s-topped $q\ge p$ with $u_q\subseteq I\mr{<s'}$.
\end{Lemma}

Proof.  Choose any $p\in\QQ_I$.  If $p$ is s-topped, take $q=p$.  Otherwise, recall that $\tp(a^*,N^*)$ is not pseudo-algebraic.  Let $\delta(y)$ be the complete formula
generating $\tp(a^*,N^*)$.  By Remark~\ref{useremark}, choose  $\phi(y,\xbar_p)$  to
be complete, include $p$,  extend $\delta(y)$, but not be pseudo-algebraic.  Then apply Lemma~\ref{alternate} with this $\phi$ to get $q\ge p$ as required.
$\qed_{\ref{s-toppedL}}$

\begin{Lemma}  \label{addN*}
For every $p\in\QQ_I$, for every endless $J\supseteq u_p$, for every $t\in u_p\cap P$, and for every  finite $C\subseteq N^*$, there is a simple extension $q$ of $p$
satisfying $u_q\subseteq J$ and $C\subseteq range[g_{1,q}(t)]$.
\end{Lemma}

Proof.  Arguing by induction on $|C|$, it suffices to prove this for $C=\{c\}$ a singleton.  Fix $p\in\QQ_I$, $J\supseteq u_p$ and $t\in u_p\cap P$.
As $t\in u_p$, $p$ is non-trivial.  Let $B=range[g_{1,p}(t)]$ and, letting $\zbar=\xbar_p\mr{\le t}$, let $\psi(y,\zbar)=\tp(cB,N^*)$.  Extend $\psi$ to
a complete formula $\phi(y,\xbar_p)$ that includes $p(\xbar_p)$.  Apply Lemma~\ref{bigHenkin} to $p$ and $\phi$, using the `Moreover clause' to require
that $g_{1,q}(x_{s,m})=c$.
$\qed_{\ref{addN*}}$

\begin{Lemma}  \label{combo}  If $p$ is s-topped and $\phi(\ybar,\xbar_p)$ is any complete formula that includes $p$, then there is a simple extension $q$ of $p$ such that
$q(\xbar_q)=\phi$.
\end{Lemma}

Proof.  Arguing by induction on $\lg(\ybar)$, we may assume $\ybar$ is a singleton.  But then, as $p$ s-topped implies $p$ non-trivial, the result follows immediately from
Lemma~\ref{bigHenkin}.
$\qed_{\ref{combo}}$

\medskip

Next, we have a series of Lemmas aimed at  proving Proposition~\ref{compatible}, which gives a sufficient condition for two
conditions to have a common extension.  For the proof,  we distinguish two cases.  The condition is on the sets $u_p$, $u_q$, and $J$, but
they collectively describe when we  need to increase the sequence described by $g_{0,p}(t^*)$.

\begin{Definition}  {\em  Suppose $u,v$ are finite subsets of $I$ and $J\subseteq I$ is endless.
We say {\em $v$ obstructs $u$ at $J$} if
\begin{enumerate}
\item  $t^*=\min(u\setminus J)$ exists and $P(t^*)$ holds;
\item  $v\subseteq J$ is non-empty;
\item  Taking $s^*=\max(v)$, we have $\neg P(s^*)$, but $s^*/E> s/E$ for every $s\in u\cap J$.
\end{enumerate}
}
\end{Definition}

\begin{Lemma} \label{ex}
Suppose $p,q\in\QQ_I$, $J$ is endless, $u_q\subseteq J$, and $w=\max(u_p\cap J)$.  If $E(w,\max(u_q))$, then $u_q$ does not obstruct
$u_p$ at $J$.  In particular, if $q$ is a simple extension of $p|J$, then $u_q$ does not obstruct $u_p$ at $J$.
\end{Lemma}

Lemmas~\ref{easypart}  and \ref{easytwo} prove the `easier half'  of Proposition~\ref{compatible} as we do not need to extend $g_0$.

\begin{Lemma}  \label{easypart} Suppose $p,q\in\QQ_I$, $J\subseteq I$ is endless, $u_q\subseteq J$, $u_p\setminus J=\{w^*\}$ is a singleton
 $p\mr J\le q$, and
$u_q$ does not obstruct $u_p$ at $J$.  Then there is $r\in\QQ_I$ with $\xbar_r=\xbar_p\cup\xbar_q$, $r\ge p$, and $r\ge q$.
Moreover, if $q$ is a simple extension of $p\mr{J}$, then $r$ is a simple extension of $p$.
\end{Lemma}

Proof.  We split into two cases, depending on whether or not $P(w^*)$ holds.
 In both cases, as the $r$ we construct will satisfy $\xbar_r=\xbar_p\cup\xbar_q$,
the primary objective is to find an appropriate complete type $r(\xbar_r)$ extending $p(\xbar_p)\cup q(\xbar_q)$.

\medskip\par\noindent{\bf Case 1.}  $\neg P(w^*)$ holds.  [Put $s^*=w^*$ to indicate this.]

\medskip\par

Write the variables of $\xbar_p$ as $y,\ybar,\zbar$, where $y=x_{s^*,0}$, $\ybar=\<x_{s^*,j}:1\le j<n_{s^*}\>$ and $\zbar=\xbar_p\mr{<s^*}$.
As $p(y,\ybar,\zbar)$ is a complete type, so is $\phi(y,\zbar):=\exists\ybar p(y,\ybar,\zbar)$.  By the Striation Constraints, $\phi(y,\zbar)$ is not pseudo-algebraic, but
$p(\ybar;y\zbar)$ is pseudo-algebraic.  Next, write the variables $\xbar_q$ as $\zbar,\wbar$ (this uses $p\mr {<s^*}=p\mr J$ and $p\mr J\le q$).

Choose any $M\in\At$ and choose $AB\subseteq M$ such that $\tp(AB,M)=q(\zbar,\wbar)$.  As $\phi(y,A)$ is not pseudo-algebraic, we can find $c\in M$
such that $M\models \phi(c,A)$, but $c\not\in\pcl(AB,M)$.  Then choose $\dbar$ from $M$ so that $\tp(c\dbar A,M)=p(y,\ybar,\zbar)$.  Necessarily,
$\dbar\subseteq\pcl(cAB,M)$.

Now write $\xbar_r$ as $y,\ybar,\zbar,\wbar$ and let $r(\xbar_r)=\tp(c\dbar AB,M)$.  It is easily checked that the Striation Constraints are maintained.
As for $\gbar_r$, put $\gbar_r=\gbar_q$, i.e., for every $t\in u_q\cap P$, $g_{0,r}(t)=g_{0,q}(t)$ and $g_{1,r}(t)=g_{1,q}(t)$.  As $s^*>s$ for every $s\in u_q$,
the functions $\gbar_r$ are as required, simply because they were for $q$.

\medskip\par\noindent{\bf Case 2.}  $P(w^*)$ holds.  [Put $t^*=w^*$ to indicate this.]
\medskip

Here, write $\xbar_p$ as $\ybar,\zbar$, where $\ybar$ is $\xbar_p\mr{=t^*}$ and $\zbar$ is $\xbar_p\mr{<t^*}$.  As $p\mr J=p\mr{<t^*}$ and $p|J\le q$,
we can write $\xbar_q$ as $\zbar,\wbar$ and we have $q(\zbar,\wbar)\vdash p(\zbar)$.
In this case, we use the function $g_{1,p}(t^*):\ybar\zbar\rightarrow N^*$ as our guide.
We know that $B=g_{1,p}(t^*)[\ybar]\subseteq N^*\setminus M^*$, $A=g_{1,p}(t^*)[\zbar]\subseteq M^*$, and $\tp(BA,N^*)=p(\ybar,\zbar)$.
As $q(\zbar,\wbar)\vdash p(\zbar)$, choose $C\subseteq M^*$ such that $\tp(AC,M^*)=q(\zbar,\wbar)$.

Now write $\xbar_r$ as $\ybar\zbar\wbar$ and let $r(\xbar_r)=\tp(BAC,N^*)$.  It is evident that the Striation Conditions are satisfied.
As for $\gbar_r$, we can put $\gbar_r(t)=\gbar_q(t)$ for every $t\in u_q\cap P$.  So, it only remains to define $g_{0,r}(t^*)$ and $g_{1,r}(t^*)$.
The latter is easy, as we used $N^*$ as our template.  That is, define $g_{1,r}(t^*)$ to be the function mapping $\ybar\zbar\wbar$ to $BAC$.

Finally, the definition of $g_{0,r}(t^*)$ is where we use our assumption that $u_q$ does not obstruct $u_p$ at $J$.  We are assuming that
$t^*=\min(u_p\setminus J)$ and $P(t^*)$ holds.  Choose $s\in u_p$ to be from the largest $E$-class in $u_p$ below $t^*$.  As $p\in\QQ_I$,
$\neg P(s)$ holds.  Let $s^*=\max(u_q)$.    As $u_q$ does not obstruct $u_p$ at $J$, $s^*/E\le s/E$.  Thus, there is no reason to extend $g_{0,p}(t^*)$,
and we simply let $g_{0,r}(t^*)=g_{0,p}(t^*)$.
$\qed_{\ref{easypart}}$

\begin{Lemma}  \label{easytwo}
 Suppose $p,q\in\QQ_I$, $J\subseteq I$ is endless, $u_q\subseteq J$,
 $p\mr J\le q$, and
$u_q$ does not obstruct $u_p$ at $J$.  Then there is $r\in\QQ_I$ with $\xbar_r=\xbar_p\cup\xbar_q$, $r\ge p$, and $r\ge q$.
Moreover, if $q$ is a simple extension of $p\mr{J}$, then $r$ is a simple extension of $p$.
\end{Lemma}

Proof.  Arguing by induction on $|u_p\setminus J|$, this follows immediately from Lemma~\ref{easypart}.
$\qed_{\ref{easytwo}}$

We now consider the `harder half' where we do need to extend $g_0$.

\begin{Lemma}  \label{hardpart}  Suppose $p,q\in \QQ_I$, $t^*=\max(u_p)$, $P(t^*)$ holds, $u_q\subseteq I\mr{<t^*}$, and $p\mr{<t^*}\le q$.
Then there is $r\in\QQ_I$, $r\ge p$, $r\ge q$, and $\max(u_r)=t^*$.
\end{Lemma}

Proof.  Let $w^*=\max(u_q)$.  If $w^*/E\le s/E$ for some $s\in u_p$, $s<t^*$, then $u_q$ would not obstruct $u_p$ at $I\mr{<t^*}$ and we would be done by Lemma~\ref{easypart}.
So assume that $w^*/E>s/E$ for every $s\in u_p$, $s<t^*$.  First, by Lemma~\ref{s-toppedL}, we may assume that $q$ is s-topped.
Arguing as in Case 2 of Lemma~\ref{easypart} write $\xbar_p$ as $\ybar\zbar$, where $\ybar$ is $\xbar_p\mr{=t^*}$ and $\zbar$ is $\xbar_p\mr{<t^*}$.
As well, write $\xbar_q$ as $\zbar,\wbar$ where $q(\zbar,\wbar)\vdash p(\zbar)$.
Again, $B=g_{1,p}(t^*)[\ybar]\subseteq N^*\setminus M^*$ and $A=g_{1,p}(t^*)[\zbar]\subseteq M^*$, where $\tp(BA,N^*)=p(\ybar,\zbar)$.
As $q(\zbar,\wbar)\vdash p(\zbar)$, choose $C\subseteq M^*$ such that $\tp(AC,M^*)=q(\zbar,\wbar)$.
As $AC$ is finite, choose $m\in\omega$ such that $AC\subseteq M_m$.  Now consider $\tp(\dbar_m/AC,M_m)$.  By applying Lemma~\ref{combo}, there is a simple extension
$q'$ of $q$
such that $\tp(\dbar_mAC,M_m)=q'(\xbar_{q'})$.  As $q'$ is a simple extension of $q$,  $\max(u_{q'})<t^*$.  Thus, by replacing $q'$ by $q$, we may additionally assume that
$\dbar_m\subseteq AC$.
Finally, choose $s^*\in I$ satisfying $w^*/E<s^*/E<t^*$ and $\neg P(s^*)$.

We are now able to define $r$.  Put $\xbar_r=\{x_{s^*,0}\}\cup\ybar\zbar\wbar$ and let $r(\xbar_r)=\tp(c_mBAC,N^*)$.  
For $t^*$, first let $g_{1,r}(t^*)$ map $\xbar_r$ onto $c_mBAC$.  As we assumed $\dbar_m\subseteq AC$, $\dbar_m\subseteq range[g_{1,r}(t^*)]$.  Finally,
let $\ell_{t^*,r}=\ell+1$, where $\ell=\ell_{t^*,p}$ and let $g_{0,r}(t^*)$ be the one-element end extension of $g_{0,p}(t^*)$ formed by $g_{0,r}(t^*)(\ell)=s^*$.
It is easily verified that $r\in\QQ_I$ is as desired.
$\qed_{\ref{hardpart}}$

\begin{Proposition}  \label{compatible} Suppose $p,q\in\QQ_I$, $J\subseteq I$ is endless, $u_q\subseteq J$, and
 $p\mr J\le q$.
Then there is $r\in\QQ_I$ with $\max(u_r)=\max(u_p)$, $r\ge p$, and $r\ge q$.
\end{Proposition}

Proof.  This follows immediately by induction on the number of $E$-classes of elements in $u_p\setminus J$, using either Lemma~\ref{easypart} or Lemma~\ref{hardpart}
at each step.
$\qed_{\ref{compatible}}$

\medskip

From this, we can easily verify that $\QQ_I$ has the c.c.c.

\begin{Lemma}\label{ccc}
$(\QQ_I,\le_\Q)$ has the c.c.c.
\end{Lemma}

Proof.  Let $\{ p_i\mcolon i< \aleph_1\}\subseteq\QQ_I$ be  a collection of
conditions.  We will find $i\neq j$ for which $p_i$ and $p_j$ are compatible.
We successively reduce this set maintaining its uncountability. By
the $\Delta$-system lemma we may assume that there is a single $u^*$
such that for all $i,j$, $u_{p_i} \cap u_{p_j} = u^*$. Further, by the
pigeonhole principle we can assume that for each $t\in u^*$, $n_{p_i,t}
= n_{p_j,t}$.   We can use pigeon-hole again to guarantee that all
the $p_i$ and $p_j$ agree on the finite set of shared variables. Furthermore, by pigeon-hole,
we may assume $\gbar_{p_i}(t)=\gbar_{p_j}(t)$ for all $t\in u^*\cap P$.  And
finally, since $I$ is $\aleph_1$-like we can choose an uncountable
set $X$ of conditions such that for $i<j$ and $p_i,p_j \in X$ all
elements of $u^*$ precede anything in any $u_{p_i}\setminus u^*$ or $u_{p_j}\setminus u^*$
and that all elements of $u_{p_i}\setminus u^*$ are less that all elements of
$u_{p_j} \setminus u^*$.

Finally, choose any $i<j$ from $X$.  Let $J=\{s\in I:s< \min(u_{p_j}\setminus u_{p_i})\}$.
By Proposition~\ref{compatible} applied to $p_i$ and $p_j$ for this choice of $J$, we conclude that $p_i$ and $p_j$ are compatible.
$\qed_{\ref{ccc}}$

\medskip

In the remainder of Section~\ref{proof1} we list the crucial ‘constraints’, which are sets of
conditions, and we prove each of them to be dense and open in $\QQ_I$. While A-C are quite similar to \cite{BLSmanymod}; the later ones depend more on this context. Before stating the first constraint we prove a lemma needed to study it.

\begin{Lemma}  \label{t-topped}  For every $p\in\QQ_I$ and every $t\in I$ such that $P(t)$ holds and $w<t$ for every $w\in u_p$, there is $q\in\QQ_I$, $q\ge p$,
with  $\max(u_q)=t$ and $u_q\cap P=(u_p\cap P)\cup\{t\}$.
\end{Lemma}

Proof.  Given $p$ and $t$ as assumed, first note that by Lemma~\ref{s-toppedL}, we may assume $p$ is s-topped.  Choose any $B$ from $M^*$ realizing $p(\xbar_p)$.
We construct $q$ as follows: $\xbar_q=\{x_{t,0}\}\cup\xbar_p$, and put $q(\xbar_q)=\tp(a^*B,N^*)$,
where $a^*$ is the `preferred element' in $N^*\setminus M^*$.  For $t'\in u_p\cap P$, take $\gbar_q(t')=\gbar_p(t')$.  Let $g_{1,q}(t)$ be the elementary map from
$\xbar_q$ onto $a^*B$, and let $g_{0,q}(t)=s$ for any choice of  $s\in u_p$ for which $s/E$ is maximal in $u_p$.
It is easily verified that $q$ is as required.
$\qed_{\ref{t-topped}}$

\medskip\par\noindent
{\bf A.  Surjectivity}
Our first group of constraints ensure that for any generic $G\subseteq\QQ_I$,
for every $(t,n)\in I\times\omega$, there is $p\in G$ such that $x_{t,n}\in \xbar_p$.
To enforce this, for
any $(t,n)\in I\times\omega$, let
$$\A_{t,n}=\{p\in\QQ_I: x_{t,n}\in\xbar_p\}$$

\begin{Claim}  \label{surjective}
For every $(t,n)\in I\times\omega$, $\A_{t,n}$ is dense and open.
\end{Claim}

Proof.  Each of these sets are trivially open.
We first show that $\A_{t,0}$ is dense for every $t\in I$.  To see this,  first recall that $\A_{\min(I),0}=\QQ_I$ by Striation constraint 1(a).
Next, fix $t\neq\min(I)$ and choose $p\in\QQ_I$ arbitrarily.
Take the endless proper initial segment $J=I_{<t}$ and let $q=p\mr J$.  Using either Lemma~\ref{s-toppedL} or Lemma~\ref{t-topped}
(depending on whether or not $P(t)$ holds) we get $q'\ge q$ with $u_{q'}\subseteq J$ and $q'\in\A_{t,0}$.  Now, using Proposition~\ref{compatible},
we get $r\ge p$ and $r\ge q'$.  As $r\in\A_{t,0}$, we have shown $\A_{t,0}$ to be dense.

Next, we prove by induction on $n$ that if $\A_{t,n}$ is dense, then so is $\A_{t,n+1}$.  But this is trivial.  Fix $t$ and choose
$p\in\QQ_I$ arbitrarily.  By our inductive hypothesis, there is $q\ge p$ with $x_{t,n}\in \xbar_q$.  If $x_{t,n+1}\in\xbar_q$, there is nothing to
prove, so assume otherwise.  Then, necessarily, $n_{q,t}=n+1$.  Let $r$ be the extension of $q$ formed by $\xbar_r=\xbar_q\cup\{x_{t,n+1}\}$
and $r(\xbar_r)$ the complete type generated by $q(\xbar_q)\cup\{x_{t,n+1}=x_{t,n}\}$.  We take $g_{0,r}=g_{0,q}$ and for every $t'\in u_q\cap P$ with $t\le t'$,
let $g_{1,r}(t')$ be the one-point extension of $g_{1,q}(t')$ satisfying $g_{1,q}(t')(x_{t,n+1})=g_{1,q}(t')(x_{t,n})$.
$\qed_{\ref{surjective}}$

Next we describe the Henkin constraints.

\medskip\par\noindent
{\bf B. Henkin}  For every $t\in I\setminus\{\min(I)\}$, finite sequence $\zbar$ from $I_{<t}\times\omega$ (in the sense of Definition~\ref{finiteseq})
and complete formula $\phi(y,\zbar)$,
$$\B_{t,\phi}=\{p\in\QQ_I:\zbar\subseteq \xbar_p \ \hbox{and {\bf either} $\neg\exists y\phi(y,\zbar)\in p$ {\bf or} $\phi(x_{s,m},\zbar)\in p$ for some $s<t,m<\omega$}\}$$

\begin{Claim}  \label{Henkin}
For every $t\in I\setminus\{\min(I)\}$ and complete $\psi(y,\zbar)$, $\B_{\psi,t}$ is dense and open.
\end{Claim}

Proof.  Fix $t$ and $\phi$.  Clearly, $\B_{t,\psi}$ is open.  To show density, choose any $p\in\QQ_I$.
By applying Lemma~\ref{s-toppedL}, we may assume $p$ is non-trivial.
By iterating Claim~\ref{surjective}, we may also assume $\zbar\subseteq\xbar_p$.
There are now two cases:  First, if $p\vdash \neg\exists y\psi(y,\zbar)$, then $p\in\B_{t,\psi}$, so assume otherwise, i.e., $p(\xbar_p)\cup\{\psi(y,\zbar)\}$ is consistent.
Choose any complete formula $\phi(y,\xbar_p)$ extending $\psi$ that includes $p$.  It follows immediately from Lemma~\ref{bigHenkin} that there  is a simple
extension $q$ of $p$ with $q\in\B_{t,\psi}$.
$\qed_{\ref{Henkin}}$

In order to ensure our generic model is `full' we need a minor variant of the Henkin constraints.

 \medskip\par\noindent {\bf C.  Fullness}
Suppose $\zbar$ is a finite sequence (in the sense of Definition~\ref{finiteseq}), $t\in I$, and a formula $\phi(y,\zbar)$ satisfies
 $\phi(y,\zbar)\vdash\theta(\zbar)$ for some complete formula $\theta(\zbar)$, but $\phi(y,\zbar)$
 is not pseudo-algebraic.

$$\CC_{\phi,t}=\{p\in\QQ_I:\ \hbox{there is}\ s>t, s\in u_p,\zbar\subseteq\xbar_p, p\vdash\phi(x_{s,0},\zbar)\}$$

\begin{Claim}  \label{fullness}  Each is $\CC_{\phi,t}$ is dense and open.
\end{Claim}

Proof.  Fix $\phi(y,\zbar)$ and $t$, and choose any $p\in\QQ_I$.  Fix any $s^*>t$.  By extending $p$ as needed, by   Claim~\ref{surjective} we may assume
$s^*\in u_p$  and $\zbar\subseteq\xbar_p$. As $\phi(y,\zbar)$ is not pseudo-algebraic, there is a complete, non-pseudo-algebraic $\psi(y,\xbar_p)$ extending $\phi$
and including $p$.   By Lemma~\ref{bigHenkin} there is a simple extension $q$ of $p$ and $s\ge s^*$ (hence $s>t$) such that $q\vdash\phi(x_{s,0},\zbar)$.
$\qed_{\ref{fullness}}$

\medskip\par\noindent
{\bf D. ${\bf g_0}$ cofinal}  We introduce two sets of conditions that guarantee that in any generic $G$, $g_0(t)$ will describe an $\omega$-chain that is cofinal in $I_{<t}$.
\medskip\par
Fix $t\in I$ such that $P(t)$ holds.
\begin{itemize}
\item For each $s<t$, $\D_{t,s}=\{p\in\QQ_I:t\in u_p\ \hbox{and for some $n$, $s/E< g_{0,p}(t)(n)/E<t$}\}$;
\item For every $n\in\omega$, $\E_{t,n}=\{p\in\QQ_I:t\in u_p\ \hbox{and}\ n\in\dom(g_{0,p}(t))\}$.
\end{itemize}

\begin{Claim}  \label{omega}
For every $t\in I$ satisfying $P(t)$, every $s<t$, and every $n\in\omega$, both $\D_{t,s}$ and $\E_{t,n}$ are dense and open.
\end{Claim}

Proof.  That each set  is open is immediate.  Fix $t$ such that $P(t)$ holds.  We first argue that $\D_{t,s}$ is dense for every $s<t$.
To see this, fix $s<t$ and choose any $p\in\QQ_I$.  By Claim~\ref{surjective} we may assume $t\in u_p$.
Choose an endless initial segment $J$
such that $t=\max(u_p\cap J)$.  Let $q=p\mr J$ and let $q_1=p\mr{<t}$.  Note that $\max(u_q)=t$.
 Let $s^*=\max(u_{q_1}\cup\{s\})$.  From the definition of $\II^*$, choose $t^*\in I$
such that $s^*<t^*<t$ and $P(t^*)$ holds.

We first find $r\ge q$ with $\max(u_r)=t$ and $r\in\D_{t,s}$.
By Lemma~\ref{t-topped}, choose $q_2\ge q_1$ with $\max(u_{q_2})=t^*$.  Apply Proposition~\ref{compatible} to
$q_2$ and $q$ to obtain an upper bound $r$ satisfying $\max(u_r)=t$.  As  $r\in\QQ_I$ while $s$ is not in the maximal $E$-class of $u_r\setminus \{t\}$,
by Condition~3(c) there is $n$ such that $s/E<g_{0,r}(t)(n)$. That is, $r\in\D_{t,s}$.  Finally, apply Proposition~\ref{compatible} again to $p$ and $r$ to get an extension of $p$ that is in $\D_{t,s}$.

Next,  we argue by induction on $n$ that each set $\E_{t,n}$
is dense.  We begin with $n=0$.  Given $t$ and any $p\in\QQ_I$, by Claim~\ref{surjective} there is $q\ge p$ with $t\in u_q$.  As $q\in\QQ_I$, Condition~3(a) implies
that $q\in\D_{t,0}$.

Finally, assume $\E_{t,n}$ is dense.  Choose any $p\in \QQ_I$.  Using the density of $\E_{t,n}$, we may assume $n\in\dom(g_{0,p}(t))$.  If $(n+1)\in\dom(g_{0,p})$
then there is nothing to prove, so assume otherwise.  That is, $\dom(g_{0,p}(t))=\{0,\dots,n\}$.  Let $s=g_{0,p}(t)(n)$.  Choose $s'$ satisfying
$s/E<s'/E<t$.  As $\D_{t,s'}$ is dense, choose $q\ge p$
with $q\in \D_{t,s'}$.  As $s'\in u_q$, $s/E$ is not in the maximal $E$-class of $u_q$ below $t$, hence $g_{0,q}(t)$ properly extends $g_{0,p}(t)$, so $q\in\E_{t,n+1}$.
$\qed_{\ref{omega}}$

 \medskip\par\noindent {\bf F.  Adjusting the level}
Suppose $t\in I$ such that $P(t)$ holds, $\zbar$ is a finite sequence (in the sense of Definition~\ref{finiteseq}) from $I_{\le t}\times\omega$, $w>t$ and $n\in\omega$.
Then $\F_{t,\zbar,w,n}$ consists of all $p\in\QQ_I$ such that $\{x_{w,n}\}\cup\zbar\subseteq \xbar_p$ and {\bf either} $p$ `says' $x_{w,n}\not\in\pcl(\zbar)$ {\bf or}
$p$ `says' $x_{w,n}=x_{s,m}$ for some $s\le t$, $m\in\omega$.

\begin{Claim}  \label{level}
For all $t,\zbar,w,n$ as above, $\F_{t,\zbar,w,n}$ is dense and open.
\end{Claim}

Proof.  As always, `open' is clear.  To establish density, choose any $p\in\QQ_I$.  By iterating Claim~\ref{surjective} we may assume $\{x_{w,n}\}\cup\zbar\subseteq \xbar_p$.
There are now two cases.  If $p$ `says' $x_{w,n}\not\in\pcl(\zbar)$ we are done, so assume otherwise. Let $\phi(y,\xbar_p)$ be the complete formula including $p$
that extends `$y=x_{w,n}'$.  Now apply Lemma~\ref{bigHenkin} to get a simple extension $q$ of $p$ with $\xbar_q=\{x_{s,m}\}\cup\xbar_p$ and $\phi(x_{s,m},\xbar_p)$.
Note that since $x_{w_n}\in\pcl(\zbar)$, we have that $\phi_t(y,\xbar_p\mr{\le t})$ is pseudo-algebraic.  Thus, $s\le t$ as required to show $q\in\F_{t,\zbar,w,n}$.
$\qed_{\ref{level}}$

\medskip

We now verify that the forcing $(\QQ_I,\le_\QQ)$ satisfies the conclusions of Proposition~\ref{forcing}.  Suppose $G\subseteq\QQ_I$ is a filter meeting every dense open subset.
Let $$X[G]=\bigcup\{p(\xbar_p):p\in G\}$$
Because of the dense subsets $\A_{t,n}$, $X[G]$  describes a complete type in the variables $\{x_{t,n}:t\in I,n\in\omega\}$.
There is a natural equivalence relation $\sim_G$ on $X[G]$ defined by
$$x_{t,n}\sim_G x_{s,m}\qquad\hbox{if and only if}\qquad X[G] \ \hbox{`says'}\ x_{t,n}=x_{s,m}$$

Let $N[G]$ be the $\tau$-structure with universe $X[G]/\sim_G$. As notation, for each $t,n$, let $a_{t,n}\in N[G]$ denote $[x_{t,n}]=x_{t,n}/\sim_G$.
 As each $p\in\QQ_I$ describes a complete (principal) formula with respect to $T$, $N[G]$ is an atomic set.  As well, it follows from Claim~\ref{Henkin}
that $N[G]\models T$.

For each $t\in I$ such that $P(t)$ holds, let $N_t=\{[x_{w,n}]:$ some $x_{s,m}\in[x_{w,n}]$ with $s<t\}$.
Similarly, for each $s\in I\setminus\{\min(I)\}$ with $\neg P(s)$, let
$N_s=\{[x_{w,n}]:w/E < s/E\}$.

By repeated use of Claim~\ref{Henkin}, each $N_t$ and $N_s$ are elementary substructures of $N[G]$.  Note that  $N_{s'}=N_s$ whenever $E(s',s)$.

Given any $(w,n)$, if there is a least $s\in I$ such that $a_{w,n}=a_{s,m}$ for some $m\in\omega$,
then we say
{\em $a_{w,n}$ is on level $s$}.  For an arbitrary $(w,n)$, a least $s$ need not exist, but it does in many cases.
 Recall that the Striation constraints imply that for every $w\in I$,  $a_{w,0}$ is on level $w$.
As well, for any $n>0$, $a_{t,n}$ is on level $t$ for any $t$ such that $P(t)$ holds.
Because of the Level constraints (group $\FF$) for any $t$ such that $P(t)$ holds, if $b\in N[G]$ and $b\in\pcl(\{a_{s,m}:s\le t,m\in\omega\},N[G])$,
$b=[x_{s',m'}]$ for some $m'\in\omega$ and $s'\le t$.

As $|I|=\aleph_1$ and the fact that each $a_{t,0}\not\in\pcl(N_{t}, N[G])$, $||N[G]||=\aleph_1$.
Finally, it follows from the density of the `Fullness conditions' that $N[G]$ is full.

More information about $N[G]$ can be gleaned from the functions $g_1$ and $g_0$.
Fix any $t$ such that $P(t)$ holds.  Let $B_t\subseteq N[G]$ consist of all elements at level exactly $t$.
Define $$g_1^*(t):B_t\cup N_t\rightarrow N^*$$
by $g_1^*(t)(a_{s,m})=g_{1,p}(t)(x_{s,m})$ for some $p\in G$.  As each  map $g_{1,p}(t)$ is elementary and $g_{1,p}(t)\subseteq g_{1,q}(t)$ whenever $q\ge p$, this is well-defined.
Because of Constraints~4(c,d), $g_1^*$ maps $N_t$ into $M^*$ and takes $B_t$ into $N^*\setminus M^*$.
Furthermore, because of Lemma~\ref{addN*},  $g_1^*$ maps onto $N^*$.  It follows that $B_t\cup N_t$ is the universe of an elementary submodel of $N[G]$
that is isomorphic to $N^*$ via $g_1^*$.  Similarly, the restriction of $g_1^*(t)$ to $N_t$ is onto, hence yields an isomorphism between $N_t$ and $M^*$.
Finally, the restriction of $g_1^*(t)$ to $B_t$ is onto $N^*\setminus M^*$.  Also, by Constraint~4(b), $g_1^*(t)(a_{t,0})=a^*$.

Similarly, because of density groups $\D$ and $\E$, the function
$$g_0^*(t):\omega\rightarrow I_{<t}$$ defined by $g_0^*(t)(i)=g_{0,p}(t)(i)$ for some $p\in G$ is well-defined.  Moreover, if we let $s_i$ denote $g_0^*(t)(i)$,
then the sequence $\<s_i:i\in\omega\>$ is cofinal in $I_{<t}$ and satisfies $s_i/E<s_{i+1}/E$ for all $i$.  By Constraint group~5 and our comments about $g_1^*(t)$
above, for every $i>0$ there is $m(i)\in\omega$, $\bbar$ from $N_{s_i}$  such that
\begin{itemize}
\item $g_1^*(t)(\bbar)=\dbar_m\subseteq M^*$;
and
\item  $g_1^*(t)(a_{s_{i},0})=c_m$
\end{itemize}
As $s_i/E<s_{i+1}/E$, $a_{s_{i},0}\in N_{s_{i+1}}$.  As well, as the Striation constraints imply that $a_{s_{i},0}$ is at level $s_{i}$, $a_{s_{i},0}\not\in N_{s_i}$.
Finally, using $g_1(t)^{-1}$ the relation $$c_m\in\pcl(\dbar_ma^*,N^*)$$
translates into $$a_{s_{i},0}\in\pcl(\bbar a_{t,0},N[G])$$

It remains to verify that $N[G]$ satisfies the three conditions of Proposition~\ref{forcing}.
(1) is handled by the Henkin constraints, most notably Claim~\ref{Henkin}.

 Towards (2), the argument just given implies that $a_{t,0}$ admits a cofinal chain in $N_t$ for every $t$ such that $P(t)$ holds.
To complete the verification of (2) we show that $a_{t,0}$ also $\rk\ \infty$-catches $N_t$ whenever $P(t)$ holds.
Fix such a $t$.  We know that $N_t\preceq N[G]$, so it is pseudo-algebraically closed.
Choose any $b\in N[G]$ such that $b\in\pcl(a_{t,0}N_t,N[G])\setminus N_t$.  Because of the Level constraints (group $\F$) we have
that $b=a_{w,n}$ for some $w\le t$.  However, if $w<t$, then we would have $b\in N_t$, which it isn't.
Thus, $b$ is of level $t$, hence $b\in B_t$ in the notation defined above.  Applying $g_1^*(t)$ to $a_{t,0}, N_t, b$ yields:
$$e\in\pcl(a^*M^*,N^*)\setminus M^*$$
where $e=g_1^*(t)(b)$.
By Data~\ref{DataA}(2) of the initial data, this implies $\rk(a^*/M^*e)<\infty$.
Translating back via $g_1^*(t)$ yields $\rk(a_{t,0}/N_tb)<\infty$.
Thus, $a_{t,0}$ $\rk\ \infty$-catches $N_t$.

It remains to verify (3) of Proposition~\ref{forcing}.  Choose a seamless
$J\subseteq I$ and let $N_J:=\bigcup\{N_t:t\in J\}\preceq N[G]$.  
Choose any  $b\in N[G]\setminus N_J$  that $\rk\ \infty$ catches
$N_J$, and we  show that $b$ has bounded effect in $N_J$.
Say $b$ is $[x_{w^*,n}]$, where necessarily $w^*\in I\setminus J$.
 By the fundamental
theorem of forcing, there is $p\in G$ such that
\begin{equation}
p\Vdash [x_{w^*,n}]_{\Gdot}\ \hbox{$\rk\ \infty$-catches}\ N_J[\Gdot].  \tag{*}
\end{equation}
Thus, among other things, $p\Vdash$ `$x_{w^*,n}\neq x_{s,m}$' for all $s\in J$, $m\in\omega$.
Choose any $s^*\in J$ such that $s^*>s$ for every $s\in u_p\cap J$. (Recall $u_p$ from just below Definition~\ref{finiteseq}.)
That $b$ has bounded effect in $N_J$ follows immediately from the following Claim.

\begin{Claim}  \label{forceit}
$p\Vdash \pcl(\{[x_{w^*,n}]\}\cup N_{<s^*}[\Gdot],N[\Gdot])\cap N_J[\Gdot]\subseteq N_{<s^*}[\Gdot]$.
\end{Claim}

Proof.  If not, then there is $q\in\QQ_I$ satisfying $q\ge p$ and a finite $A=\{x_{w_i,m_i}:i<k\}\subseteq I_{<s^*}\times\omega$ such that, letting
$A_\Gdot=\{[x_{w_i,m_i}]_{\Gdot}: i<k\}$,
$$q\Vdash \pcl(A_\Gdot[x_{w^*,n}]_\Gdot,N[\Gdot])\cap N_J[\Gdot]\not\subseteq N_{<s^*}[\Gdot]$$
Without loss, we may assume that for each $x_{w_i,m_i}\in A$, then $w_i\in u_q$.
As $J$ is seamless, by Lemma~\ref{seamless}, choose an automorphism $\pi$ of $(I,<,E,P)$ such that
$\pi\mr{\ge\min(u_p\setminus J)}=id$; $\pi(t^*)=t^*$; $\pi\mr{u_p}=id$; $\pi\mr{u_q\cap I_{<s^*}}=id$, but $\pi(s^*)\not\in J$.
By Lemma~\ref{extendauto}, $\pi$ extends to an automorphism $\pi'$ of $\QQ_I$ given by $x_{t,m}\mapsto x_{\pi(t),m}$.
By our choice of $\pi$, $\pi'(p)=p$.  Whereas $\pi'(q)$ need not equal $q$, we do have $p\le\pi'(q)$.
\begin{equation}
\pi'(q)\Vdash \pcl(A_\Gdot[x_{w^*,n}]_\Gdot,N[\Gdot])\cap N_{\pi(J)}[\Gdot]\not\subseteq N_{<\pi(s^*)}[\Gdot] \tag{**}
\end{equation}

We argue that this forcing statement contradicts the just defined $(*)$.
To see this, choose a generic $H\subseteq\QQ_I$ with $\pi'(q)\in H$.  As $p\le\pi'(q)$, we also have $p\in H$.
As above, let $N[H]\in\At$ be the structure with universe $X/\sim_H$.  Put
$N_J[H]:=\bigcup\{N_t[H]:t\in J\}\preceq N[H]$ and
 $N_{\pi(J)}[H]:=\bigcup\{N_t[H]: t\in\pi(J)\}\preceq N[H]$.
 As notation, let $b_H:=[x_{w^*,n}]_H$ and $A_H:=\{[x_{w_i,m_i}]_H:i<k\}$.
Now:

\begin{itemize}

\item  $A_H\subseteq N_J[H]$.
\item  Applying $\pi'$ to the statement $(*)$, along with $\pi'(p)=p$ and $\pi(w^*)=w^*$ yields
$$p\Vdash [x_{w^*,n}]_{\Gdot}\ \hbox{$\rk\ \infty$-catches}\ N_{\pi(J)}[\Gdot].$$
As $p\in H$ and recording only half of the definition of `$\rk\ \infty$-catches' yields
$$\rk(b_H/N_{\pi(J)},N[H])=\infty$$

\item  From $(**)$, choose $e\in N[H]$ such that
 $e\in\pcl(A_Hb_H,N[H])$ and
 $e\in N_{\pi(J)}[H]$, but
$e\not\in N_{<\pi(s^*)}[H]$. Thus, since $J\subseteq I_{<\pi(s^*)}$, $e\not\in N_J[H]$.

\end{itemize}

However, since  $\{e\}\cup N_J[H]\subseteq N_{\pi(J)}[H]$, we conclude $\rk(b_H/N_J[H]e,N[H])=\infty$.
Combining this with $e\in\pcl(N_J[H]b_H,N[H])$ and $e\not\in N_J[H]$, we contradict `$b_H$ $\rk\ \infty$-catches $N_J[H]$.'
$\qed_{\ref{forceit}}$

\medskip
Finally, as Claim~\ref{forceit} holds for any sufficiently large $s^*\in J$, $b=a_{w^*,n}$ has bounded effect in $N_J$.
This establishes (3), and thus concludes the proof of Proposition~\ref{forcing}. $\qed_{\ref{forcing}}$

\subsection{Many non-isomorphic models in $\At$}  \label{final}

We continue to work under the assumption of Data~\ref{DataA} and the notation there.  
With Proposition~\ref{mainpart} below,  we prove the existence of $2^{\aleph_1}$ non-isomorphic atomic models, each of size $\aleph_1$,
 under the assumption that a countable, transitive model $(M,\epsilon)$ of ZFC exists.  
 The main theorem of this section, Theorem~\ref{bigone}, follows easily from this.

\begin{Proposition}  \label{mainpart}   Assume that a countable, transitive model of ZFC exists.  If we have an instance of Data~\ref{DataA},
then there are atomic models ($N_X:X\subseteq \omega_1$) such that $N_X\not\cong N_Y$ whenever $X\triangle Y$ is stationary.
\end{Proposition}

\begin{proof}  
Choose $N^*,M^*,a^*,\cbar_m,c_m$ witnessing Data~\ref{DataA} and fix  a countable, transitive model $(M,\epsilon)$ of ZFC 
containing these sets, along with $T$ and $\tau$.
We begin by working inside $M$.   In particular, choose $S\subseteq\omega_1^M$ such that $$(M,\epsilon)\models
`\hbox{$S$ is stationary/costationary'}$$  Now perform Construction~\ref{IS} inside $M$ to obtain $I=(I^S,<,P,E)\in\II^*$.

Next, we force with the c.c.c.\ poset $\QQ_{I^S}$ and find $(M[G],\epsilon)$, where $G$ is a generic subset of $\QQ_{I^S}$.
As the forcing is c.c.c., it follows that all cardinals are preserved, as well as
stationarity.  Thus, $\omega_1^{M[G]}=\omega_1^M$ and
$(M[G],\epsilon)\models `\hbox{$S$ is stationary/costationary'}$.  As well, $(I^S)^{M[G]}=I^S$.
According to Proposition~\ref{forcing}, inside $M[G]$ there is an atomic, full $N_I\models T$ that is striated according to $(I^S,<,P,E)$.  Write the universe of $N_I$ as
$\{a_{t,n}:t\in I^S,n\in\omega\}$.  Inside $M[G]$ we have the mapping
$\alpha\mapsto J_\alpha$ given by Construction~\ref{IS}.  For every $\alpha\in\omega_1^{M[G]}$,
let  $N_\alpha$ be the $\tau$-substructure of $N_I$ with universe $\{a_{t,n}:t\in J_\alpha, n\in\omega\}$.  It follows from Proposition~\ref{forcing} and Construction~\ref{IS} that
for every non-zero $\alpha\in\omega_1^{M[G]}$,
\begin{itemize}
\item  $N_\alpha\preceq N_I$;
\item If $\alpha\in S$, then $I^S\setminus J_\alpha$ has a least $E$-class $t/E$ which is topped with $t^*=top(t/E)$, and $a_{t^*,0}$ both $\rk\ \infty$-catches and admits a cofinal chain in $N_\alpha$; and

\item If $\alpha\not\in S$, then every $b\in N_I\setminus N_\alpha$ that $\rk\ \infty$-catches $N_\alpha$ has bounded effect in $N_\alpha$.
\end{itemize}
Now, still working inside $M[G]$, we form a 3-sorted structure $N^*$ that encodes this information.  The language of $N^*$\footnote{Note the inclusion of $F$!} will be
$$\tau^*=\tau\cup\{U,V,W,<_U,<_V,P,E,R_1,R_2,F\}$$
$N^*$ is the $\tau^*$-structure in which
\begin{itemize}
\item $\{U,V,W\}$ are unary predicates that partition the universe;
\item  $(U^{N^*},<_U)$ is  $(\omega_1^{M[G]},<)$;
\item  $(V^{N^*},<_V,P,E)$ is $(I^S,<,P,E)$;
\item  $W^{N^*}$ is $N_I$ (the functions and relations in $\tau$ only
act on the $W$-sort);
\item  $R_1\subseteq U\times V$, with $R_1(\alpha,t)$ holding if and only if
$t\in J_\alpha$;
\item $R_2\subseteq U\times W$, with $R_2(\alpha,b)$ holding if and only if
$b\in N_\alpha$; and
\item  $F:W\times U\times\omega\rightarrow V$ satisfies:  For every $b$ and for every limit ordinal $\alpha$, $\<F(b,\alpha,n):n\in\omega\>$ is a strictly increasing, cofinal sequence in
$J_\alpha$.
\end{itemize}

Note that because of $E$, $S\subseteq\omega_1^{M[G]}$ is an $\tau^*$-definable subset of the $U$-sort of $N^*$.  Also, on the $W$-sort, the relation `$b\in\pcl(\abar)$' is definable by an infinitary $\tau^*$-formula.  Thus,
the relations `$b$ $\rk\ \infty$-catches $N_\alpha$', `$b$ has bounded effect in $N_\alpha$', and `$\<F(b,\alpha,n):n\in\omega\>$ witnesses that $b$ admits a cofinal chain'
are all infinitarily $\tau^*$-definable subsets of $U\times W$.

By construction, $N^*\models\psi$, where the infinitary $\psi$ asserts:
\begin{quotation}  
\noindent For every non-zero $\alpha\in U$, {\bf either} there is an element of $W^{N^*}$ that $\rk\ \infty$-catches $N_\alpha$ and admits a cofinal chain in $N_\alpha$ {\bf or}
every element of $W^{N^*}$ that $\rk\ \infty$-catches $N_\alpha$ has bounded effect in $N_\alpha$.
\end{quotation}

Fix
an infinitary $\tau^*$-formula $\theta(x)$ such that for $x$ from the $U$-sort, $\theta(x)$ holds if and only if there exists $b\in N_I\setminus N_{J_x}$
that $\rk\ \infty$-catches and has a cofinal chain in $N_{J_x}$.
Thus, for  $\alpha\in\omega_1^{M[G]}$ we have  $$N^*\models\theta(\alpha)\quad \Longleftrightarrow\quad \alpha\in S$$
Fix a countable fragment $L_A$ of $L_{\omega_1,\omega}(\tau^*)$ to include the formulas mentioned above, along with infinitary formulas ensuring
$\tau$-atomicity.

Now, we switch our attention to $V$.  By applying Theorem~3.13 of \cite{BLSmanymod},
which is proved by the method of iterated $M$-normal ultrapowers,  to
$(M[G],\epsilon)$, $L_A$, and $N^*$,   we obtain (in $V$!)  a family $(M_X,E)$
of elementary extensions of $(M[G],\epsilon)$, each of size $\aleph_1$, indexed
by subsets $X\subseteq\omega_1$ ($=\omega_1^V$).  Each of these models of
ZFC has an $L^*$-structure, which we call $N^*_X$ inside it.  As well, for each $X\subseteq\omega_1$, there is a continuous,
strictly increasing mapping $t_X:\omega_1\rightarrow U^{N^*_X}$ with the property
that $$N^*_X\models\theta(t_X(\alpha))\quad\Longleftrightarrow\quad\alpha\in X$$
Let $(I^X,<,E,P)$ be the `$V$-sort' of $N^*_X$.  Clearly, each $I^X\in\II^*$.

Finally, the $W$-sort of each $\tau^*$-structure $N^*_X$ is an $\tau$-structure, striated by $I^X$.
We call this `reduct' $N_X$.  Note that by our choice of $L_A$ and the fact
that $N^*_X\succeq_{L_A} N^*$, we know that every $\tau$-structure $N_X$ is an atomic
model of $T$, which is easily seen to be of cardinality $\aleph_1$.

Thus, it suffices to prove that 
 there is no
$\tau$-isomorphism $f:N_X\rightarrow N_Y$ whenever $X\triangle Y$ is stationary.
For this, choose  $X,Y\subseteq\omega_1$ such that $X\setminus Y$ is stationary and by way of contradiction assume that $f:N_X\rightarrow N_Y$ were an $\tau$-isomorphism.
Each of $N_X,N_Y$ has its `expansion' to $L^*$-structures $N^*_X$ and $N^*_Y$, respectively.  As notation, for each $\alpha\in\omega_1^V$,
let $N^X_\alpha$ and $N^Y_\alpha$ denote $\tau$-elementary substructures
with universes of $R_2(t_X(\alpha),N^*_X)$ and $R_2(t_Y(\alpha),N^*_Y)$,
respectively.

Next, choose a club $C_0\subseteq\omega_1$ such that for every $\alpha\in C_0$:
\begin{itemize}
\item $\alpha$ is a limit ordinal;
\item  The restriction of $f:N_\alpha^X\rightarrow N_\alpha^Y$ is a $\tau$-isomorphism.
\end{itemize}

Put $C:=\lim(C_0)$.
As $C$ is club and $(X\setminus Y)$ is stationary, choose $\alpha$ in their intersection.  Fix a strictly increasing $\omega$-sequence $\<\alpha_n:n\in\omega\>$ of elements from $C_0$ converging to $\alpha$.
As $\alpha\in X$, we can choose an element $b\in N_X\setminus N^X_\alpha$
and a strictly increasing sequence $\<s_m:m\in\omega\>$ converging to $\alpha$ such that
$b$  $\rk\ \infty$-catches $N^X_\alpha$ and for every $m\in\omega$
$$\pcl(N^X_{s_m}\cup\{b\})\cap N^X_{s_{m+1}}\not\subseteq N^X_{s_m}.$$
As the sets $J_{\alpha_n}$ are all proper initial segments of $J_\alpha$ with $\bigcup J_{\alpha_n}=J_\alpha$, there is an integer $k$
such that for all $n\ge k$, there is an integer $m(n)$ such that $s_m\in J_{\alpha_n}$, but $s_{m+1}\not\in J_{\alpha_n}$.
Thus, for any $n\ge k$,
$$\pcl(N^X_{\alpha_n}\cup\{b\})\cap N^X_\alpha\not\subseteq N^X_{\alpha_n}.$$

But now, as `$b\in\pcl(\abar)$' is preserved under $\tau$-isomorphisms and $f[N^X_{\alpha_n}]=N^Y_{\alpha_n}$ setwise,
we have that $f(b)$   $\rk\ \infty$-catches $N^Y_\alpha$, but for every
$n\ge k$,
$$\pcl(N^Y_{\alpha_n}\cup\{f(b)\})\cap N^Y_\alpha\not\subseteq N^Y_{\alpha_n}.$$
From this, as $\pcl$ is finitely based, it follows easily that for every
$s^*\in J_\alpha$,  there is $s\in J_\alpha$, $s>s^*$ such that
$$\pcl(N^Y_{s}\cup\{f(b)\},N_Y)\cap N^Y_{\alpha}\neq N^Y_{s}$$
That is, $f(b)$ does not have bounded effect in $N^Y_\alpha$.
As $\alpha\not\in Y$, we obtain a contradiction from  $N^*_Y\models\neg\theta(t_Y(\alpha))$ and $N^*_Y\models\psi$.
\end{proof}  

\begin{Theorem}  \label{bigone}  If $\At$ has $<2^{\aleph_1}$ non-isomorphic atomic models of size $\aleph_1$, then $\At$ is ranked.
\end{Theorem}

\begin{proof}  Assume that $\At$ is not ranked.   By Proposition~\ref{charinfty} we obtain a witness to Data~\ref{DataA}.
As the proof of Proposition~\ref{mainpart} is finite, the hypotheses there can be weakened to the existence of
a countable, transitive model of a large enough, finite subset of ZFC.  However, the existence of such a countable, transitive model
is provable from ZFC itself (using the Reflection Theorem).
Thus, the existence of $2^{\aleph_1}$ non-isomorphic atomic models of size $\aleph_1$ is provable in ZFC alone.
\end{proof}

\end{document}